\documentclass[leqno]{amsart}%
\usepackage{float}
\usepackage{amsmath}
\usepackage{latexsym}
\usepackage{amsfonts}
\usepackage{amssymb}
\usepackage{graphicx}%
\theoremstyle{plain}
\newtheorem{theorem}{Theorem}
\newtheorem{corollary}[theorem]{Corollary}
\newtheorem{lemma}[theorem]{Lemma}

\theoremstyle{definition}

\numberwithin{equation}{section}
\numberwithin{theorem}{section}

\DeclareSymbolFont{letters}{OML}{cmm}{m}{it}
\DeclareSymbolFont{bletters}{OML}{cmm}{b}{it}
\DeclareMathSymbol{\bphi}{\mathord}{bletters}{"1E}
\DeclareMathSymbol{\bvarphi}{\mathord}{bletters}{"27}
\DeclareMathSymbol{\bvartheta}{\mathord}{bletters}{"23}
\DeclareMathSymbol{\bmu}{\mathord}{bletters}{"16}
\DeclareMathSymbol{\bpsi}{\mathord}{bletters}{"20}
\DeclareMathSymbol{\bnu}{\mathord}{bletters}{"17}
\DeclareMathSymbol{\balpha}{\mathord}{bletters}{"0B}
\DeclareMathSymbol{\bbeta}{\mathord}{bletters}{"0C}
\DeclareMathSymbol{\blambda}{\mathord}{bletters}{"15}
\DeclareMathSymbol{\bomega}{\mathord}{bletters}{"21}
\DeclareMathSymbol{\bkappa}{\mathord}{bletters}{"14}
\DeclareMathSymbol{\bLambda}{\mathord}{bletters}{"03}
\DeclareMathSymbol{\bdelta}{\mathord}{bletters}{"0E}
\DeclareMathSymbol{\bDelta}{\mathord}{bletters}{"01}
\DeclareMathSymbol{\bxi}{\mathord}{bletters}{"18}
\DeclareMathSymbol{\bXi}{\mathord}{bletters}{"04}
\DeclareMathSymbol{\bvareps}{\mathord}{bletters}{"22}
\begin{document}
\title[H\"{o}lder index for superprocesses at a given point]{H\"{o}lder index at a given point for density states\vspace{10pt}\\of super\thinspace-\thinspace$\alpha$\thinspace-\thinspace stable motion of
index $1+\beta$\vspace{8pt}}
\author[Fleischmann]{Klaus Fleischmann}
\address{Weierstrass Institute for Applied Analysis and Stochastics, Leibniz Institute
in Forschungsverbund Berlin e.V., Mohrenstr.\ 39, D--10117 Berlin, Germany}
\email{fleischm@wias-berlin.de}
\author[Mytnik]{Leonid Mytnik}
\address{Faculty of Industrial Engineering and Management, Technion Israel Institute of
Technology, Haifa 32000, Israel}
\email{leonid@ie.technion.ac.il}
\urladdr{http://ie.technion.ac.il/leonid.phtml}
\author[Wachtel]{Vitali Wachtel}
\address{Mathematical Institute, University of Munich, Theresienstrasse 39, D--80333
Munich, Germany}
\email{wachtel@mathematik.uni-muenchen.de}
\thanks{Extended version of November 24, 2010 of the WIAS preprint No.\thinspace1385 of
December 12, 2008,\quad ISSN 0946-8633,\quad fixed87.tex}
\thanks{Supported by the German Israeli Foundation for Scientific Research and
Development, Grant No. G-807-227.6/2003}
\thanks{Corresponding author: Klaus Fleischmann}
\thanks{Running head: H\"{o}lder index for superprocesses at a given point}
\maketitle

\thispagestyle{empty}

\setcounter{page}{0}\newpage

\begin{quotation}
\noindent\textsc{Abstract.} A H\"{o}lder regularity index at \emph{given
points}\/ for density states of $(\alpha,1,\beta)$-superprocesses with
$\alpha>1+\beta$ is determined. It is shown that this index is strictly
greater than the optimal index of \emph{local}\/ H\"{o}lder continuity for
those density states.
\end{quotation}

\section{Introduction and statement of results}

For $\,0<\alpha\leq2$\thinspace\ and $\,1+\beta\in(1,2),$\thinspace\ the
so-called $(\alpha,d,\beta)$\emph{-superprocess}\/ $X=\{X_{t}:\,t\geq
0\}$\thinspace\ in $\mathsf{R}^{d}$ is a finite measure-valued process related
to the log-Laplace equation%
\begin{equation}
\frac{\mathrm{d}}{\mathrm{d}t}u\ =\
\bDelta
_{\alpha}u\,+au-\,bu^{1+\beta}, \label{logLap}%
\end{equation}
where $\,a\in\mathsf{R}$\thinspace\ and $\,b>0$\thinspace\ are any fixed
constants. Its underlying motion is described by the fractional Laplacian $%
\bDelta
_{\alpha}:=-(-%
\bDelta
)^{\alpha/2}$\thinspace\ determining a symmetric $\alpha$--stable motion in
$\mathsf{R}^{d}$ of index $\,\alpha\in(0,2]$\thinspace\ (Brownian motion if
$\,\alpha=2),$\thinspace\ whereas its continuous-state branching mechanism%
\begin{equation}
v\,\mapsto\,-av+bv^{1+\beta},\quad v\geq0, \label{not.Psi}%
\end{equation}
belongs to the domain of attraction of a stable law of index $\,1+\beta
\in(1,2)$\thinspace\ (the branching is critical if $\,a=0).$

\textrm{F}rom now on we assume that $d<\frac{\alpha}{\beta}\,.$\thinspace
\ Then $\,X$\thinspace\ has a.s. \emph{absolutely continuous states}\/
$\,X_{t}(\mathrm{d}x)$\thinspace\ at fixed times $\,t>0$\thinspace
\ (cf.\ Fleischmann \cite{Fleischmann1988.critical} with the obvious changes
for $a\neq0$). Moreover, as is shown in Fleischmann, Mytnik, and Wachtel
\cite[Theorem~1.2(a),(c)]{FleischmannMytnikWachtel2010.optimal.Ann}, there is
a \emph{dichotomy}\/ for their density function (also denoted by $X_{t})$:
There is a continuous version $\tilde{X}_{t}$ of the density function if
$\,d=1$ and $\,\alpha>1+\beta,$\thinspace\ but otherwise the density function
$X_{t}$ is locally unbounded on open sets of positive $X_{t}(\mathrm{d}%
x)$-measure. (The case $\alpha=2$ had been derived earlier in Mytnik and
Perkins \cite{MytnikPerkins2003}.) In the case of continuity, H\"{o}lder
regularity properties of $\tilde{X}_{t}$ had been studied in
\cite{FleischmannMytnikWachtel2010.optimal.Ann}, too.

Let us first recall the notion of an optimal H\"{o}lder index at a point (see
e.g. Jaffard \cite{Jaffard1999}). We say a function $f$ is H\"{o}lder
continuous with index $\eta\in(0,1]$ \emph{at the point}\/ $x$ if there is an
open neighborhood $U(x)$ of $\,x$ and a constant $C$ such that%
\begin{equation}
\bigl|f(y)-f(x)\bigr|\,\leq\,C\,|y-x|^{\eta}\quad\text{for all}\,\ y\in U(x).
\end{equation}
The \emph{optimal H\"{o}lder index}\/ $H(x)$ of $f$ at the point $x$ is
defined as%
\begin{equation}
H(x)\,:=\,\sup\bigl\{\eta\in(0,1]:\ f\ \text{is H\"{o}lder continuous at
}x\ \text{with index }\eta\bigr\},
\end{equation}
and set to $0$ if $f$ is not H\"{o}lder continuous at $x.$

Going back to the continuous (random) density function $\tilde{X}_{t\,}%
,$\thinspace\ in what follows, $H(x)$ will denote the (random) optimal
H\"{o}lder index of $\tilde{X}_{t}$ at $x\in\mathsf{R}.$ In
\cite[Theorem~1.2(a),(b)]{FleischmannMytnikWachtel2010.optimal.Ann}, the
so-called \emph{optimal index}\/ for \emph{local}\/ H\"{o}lder continuity of
$\,\tilde{X}_{t}$ had been determined by $\,$%
\begin{equation}
\eta_{\mathrm{c}}\,:=\,\frac{\alpha}{1+\beta}-1\,\in\,(0,1). \label{eta_c}%
\end{equation}
This means that in any non-empty open set $U\subset\mathsf{R}$ with
$X_{t}(U)>0$ one can find (random) points $x$ such that $H(x)=\eta
_{\mathrm{c}\,}.$\thinspace\ This however left unsolved the question whether
there are points $x\in U$ such that $H(x)>\eta_{\mathrm{c}\,}.$\thinspace

\hspace{-3pt}The \emph{purpose}\/ of this note is to verify the following
theorem conjectured in \cite[Section~1.3]%
{FleischmannMytnikWachtel2010.optimal.Ann}. To formulate it, let
$\mathcal{M}_{\mathrm{f}}$ denote the set of finite measures on $\mathsf{R}%
^{d},\ $and $\,B_{\epsilon}(x)$ the open ball of radius $\epsilon>0$ around
$x\in\mathsf{R}^{d}.$

\begin{theorem}
[\textbf{H\"{o}lder continuity at a given point}]\label{T.prop.dens.fixed}%
\hspace{-1pt}Fix $\,\,t>0,\,\ z\in\mathsf{R},$\thinspace\ and $\,X_{0}=\mu
\in\mathcal{M}_{\mathrm{f}\,}.$\thinspace\ Let\/ $\,d=1$\thinspace\ and
$\,\alpha>1+\beta.$\thinspace\

\begin{description}
\item[(a) (H\"{o}lder continuity at a given point)] For each $\,\eta
>0$\thinspace\ satisfying
\[
\eta\,<\,\bar{\eta}_{\mathrm{c}}\,:=\,\min\Big\{\frac{1+\alpha}{1+\beta
}-1,\,1\Big\},
\]
with probability one, the continuous version $\tilde{X}_{t}$ of the density is
H\"{o}lder continuous of order $\,\eta$ at the point $z:$%
\[
\sup_{x\in B_{\epsilon}(z),\,x\neq z}\frac{\bigl|\tilde{X}_{t}(x)-\tilde
{X}_{t}(z)\bigr|}{|x-z|^{\eta}}\,<\,\infty,\quad\epsilon>0.
\]

\item[(b) (Optimality of $\bar{\eta}_{\mathrm{c}}$)] If additionally
$\beta>(\alpha-1)/2$, then\ with probability one for any $\,\epsilon>0,$%
\[
\sup_{x\in B_{\epsilon}(z),\,x\neq z}\frac{\bigl|\tilde{X}_{t}(x)-\tilde
{X}_{t}(z)\bigr|}{|x-z|^{\bar{\eta}_{\mathrm{c}}}}\,=\,\infty\quad
\text{whenever}\,\ X_{t}(z)>0.
\]
{}
\end{description}
\end{theorem}

Theorem~\ref{T.prop.dens.fixed}(b) states the optimality of $\,\bar{\eta
}_{\mathrm{c}}$ in the case $\beta>(\alpha-1)/2$. But it is easy to see that
the opposite case $\beta\leq(\alpha-1)/2$\thinspace\ implies that $\bar{\eta
}_{\mathrm{c}}=1.$ Therefore the optimality of $\,\bar{\eta}_{\mathrm{c}}$
follows here automatically from the definition of $\,H(z).$ But opposed to the
local unboudedness of the ratio \textbf{$\frac{|\tilde{X}_{t}(x)-\tilde{X}%
_{t}(z)|}{|x-z|^{\bar{\eta}_{c}}}$} in the case $\beta>(\alpha-1)/2$, we
\emph{conjecture}\/ that $\tilde{X}_{t}$ is even Lipschitz continuous at the
given $z$ for $\beta<(\alpha-1)/2$.

Since $\,\eta_{\mathrm{c}}<\bar{\eta}_{\mathrm{c}\,},$\thinspace\ at each
given point $z\in\mathsf{R}$ the density $\tilde{X}_{t}$ allows some
H\"{o}lder exponents $\eta$ larger than $\,\eta_{\mathrm{c}\,},$%
\thinspace\ the optimal H\"{o}lder index for local domains. Thus,
Theorem~\ref{T.prop.dens.fixed} nicely complements the main result of
\cite{FleischmannMytnikWachtel2010.optimal.Ann}.

The full program however would include proving that for any $\eta\in
(\eta_{\mathrm{c}\,},\bar{\eta}_{\mathrm{c}})$ with probability one there are
(random) points $x\in\mathsf{R}$ such that the optimal H\"{o}lder index $H(x)$
of $\,\tilde{X}_{t}$ at $x$ is exactly $\eta.$ Moreover, one would like to
establish the \emph{Hausdorff dimension}, say $D(\eta)$, of the (random) set
$\,\bigl\{x:\,H(x)=\eta\bigr\}.$\thinspace\ The function $\,\eta\mapsto
D(\eta)$\thinspace\ then reveals the so-called \emph{multifractal spectrum}\/
related to the optimal H\"{o}lder index at points. As we already mentioned in
\cite[Conjecture~1.4]{FleischmannMytnikWachtel2010.optimal.Ann}, we
\emph{conjecture} that%
\begin{equation}
\lim_{\eta\downarrow\eta_{\mathrm{c}}}\,D(\eta)\,=\,0\quad\text{and}\quad
\lim_{\eta\uparrow\bar{\eta}_{\mathrm{c}}}\,D(\eta)\,=\,1\,\ \text{a.s.}%
\end{equation}
The investigation of such multifractal spectrum is left for future work.

The multifractal spectrum of random functions and measures has attracted
attention for many years and has been studied for example in Dembo et
al.\ \cite{DemboPeresRosenZeitouni2001}, Durand \cite{Durand2009}, Hu and
Taylor \cite{HuTaylor2000}, Klenke and M\"{o}rters \cite{KlenkeMoerters2005},
Le\thinspace Gall and Perkins \cite{LeGallPerkins1995}, M\"{o}rters and Shieh
\cite{MoertersShieh2004} and Perkins and Taylor \cite{PerkinsTaylor1998}. The
multifractal spectrum of singularities that describe the Hausdorff dimension
of sets of different H\"{o}lder exponents of functions was investigated for
deterministic and random functions in Jaffard
\cite{Jaffard1999,Jaffard2000,Jaffard2004} and Jaffard and Meyer
\cite{JaffardMeyer1996}.

Note also that\ in the case $\,\alpha=2$ $\,$for the optimal exponents
$\eta_{\mathrm{c}}$ and $\bar{\eta}_{\mathrm{c}}$ we have
\begin{equation}
\eta_{\mathrm{c}}\downarrow0\quad\text{and}\quad\bar{\eta}_{\mathrm{c}%
}\downarrow\tfrac{1}{2}\quad\text{as}\,\ \beta\uparrow1,
\end{equation}
whereas for continuous super-Brownian motion ($\beta=1)$ one would have
$\,\eta_{\mathrm{c}}=\frac{1}{2}=\bar{\eta}_{\mathrm{c}\,}.$\thinspace\ This
discontinuity reflects the essential differences between continuous and
discontinuous super-Brownian motion concerning H\"{o}lder continuity
properties of density states, as discussed already in \cite[Section
1.3]{FleischmannMytnikWachtel2010.optimal.Ann}.\smallskip

After some preparation in the next section, the proof of
Theorem~\ref{T.prop.dens.fixed}(a),(b) will be given in Sections~\ref{sec:6}
and \ref{S.opt}, respectively.

\section{Some proof preparation}

Let $\,p^{\alpha}$\thinspace\ denote the continuous $\alpha$--stable
transition kernel related to the fractional Laplacian $\,%
\bDelta
_{\alpha}=-(-%
\bDelta
)^{\alpha/2}$\thinspace\ in $\mathsf{R}^{d},$ and $S^{\alpha}$ the related
semigroup. Fix $\,X_{0}=\mu\in\mathcal{M}_{\mathrm{f}\,}\backslash\{0\}.$

First we want to recall the \emph{martingale decomposition}\/ of the
$(\alpha,d,\beta)$-superprocess $X$ (valid for any $\alpha,d,\beta;$ see,
e.g., \cite[Lemma~1.6]{FleischmannMytnikWachtel2010.optimal.Ann}): For all
sufficiently smooth bounded non-negative functions $\varphi$ on $\mathsf{R}%
^{d}$ and $t\geq0,$%
\begin{equation}%
\displaystyle
\left\langle X_{t},\varphi\right\rangle \,=\,\left\langle \mu,\varphi
\right\rangle +\int_{0}^{t}\!\mathrm{d}s\,\left\langle X_{s},%
\bDelta
_{\alpha}\varphi\right\rangle +M_{t}(\varphi)+a\,I_{t}(\varphi)
\label{mart.dec}%
\end{equation}
with discontinuous martingale $\,$%
\begin{equation}
t\,\mapsto\,M_{t}(\varphi)\,:=\,\int_{(0,t]\times\mathsf{R}^{d}\times
\mathsf{R}_{+}}\!\tilde{N}\!\left(  _{\!_{\!_{\,}}}\mathrm{d}(s,x,r)\right)
r\,\varphi(x) \label{mart}%
\end{equation}
and increasing process $\,$%
\begin{equation}
t\,\mapsto\,I_{t}(\varphi)\,:=\,\int_{0}^{t}\!\mathrm{d}s\,\left\langle
X_{s},\varphi\right\rangle \!. \label{incr}%
\end{equation}
Here $\tilde{N}:=N-\hat{N},$ where $N\!\left(  _{\!_{\!_{\,}}}\mathrm{d}%
(s,x,r)\right)  $ is a random measure on $(0,\infty)\times\mathsf{R}^{d}%
\times(0,\infty)$ describing all the jumps $\,r\delta_{x}$\thinspace\ of\/
$\,X$\thinspace\ at times $\,s$\thinspace\ at sites\thinspace\ $x$%
\thinspace\ of size $\,r$\thinspace\ (which are the only discontinuities of
the process $X).$ Moreover, $\,$%
\begin{equation}
\hat{N}\!\left(  _{\!_{\!_{\,}}}\mathrm{d}(s,x,r)\right)  \,=\,\varrho
\,\,\mathrm{d}s\,X_{s}(\mathrm{d}x)\,r^{-2-\beta}\mathrm{d}r \label{decomp}%
\end{equation}
is the compensator of $\,N,$ where $\,\varrho:=b$\thinspace$(1+\beta
)\beta/\Gamma(1-\beta)$\thinspace\ with $\Gamma$ denoting the Gamma function.

Recall that we assumed $\,d<\frac{\alpha}{\beta}$\thinspace,\thinspace\ and
fix $\,t>0.$\thinspace\ Then the random measure $\,X_{t}(\mathrm{d}%
x)$\thinspace\ is a.s.\ absolutely continuous. \textrm{F}rom the Green's
function representation related to (\ref{mart.dec}) (see, e.g., \cite[(1.9)]%
{FleischmannMytnikWachtel2010.optimal.Ann}) we obtain the following
representation of a version of the density function of $\,X_{t}(\mathrm{d}x)$
(see, e.g., \cite[(1.12)]{FleischmannMytnikWachtel2010.optimal.Ann}):%
\begin{equation}%
\begin{array}
[c]{l}%
\displaystyle
X_{t}(x)\,=\,\mu\!\ast\!p_{t}^{\alpha}\,(x)\,+\,\int_{(0,t]\times
\mathsf{R}^{d}}\!M\!\left(  _{\!_{\!_{\,}}}\mathrm{d}(s,y)\right)
p_{t-s}^{\alpha}(y-x)\vspace{6pt}\\%
\displaystyle
+\,a\!\int_{(0,t]\times\mathsf{R}^{d}}\!I\!\left(  _{\!_{\!_{\,}}}%
\mathrm{d}(s,y)\right)  p_{t-s}^{\alpha}(y-x)=:\,Z_{t}^{1}(x)+Z_{t}%
^{2}(x)+Z_{t}^{3}(x),\quad x\in\mathsf{R}^{d},
\end{array}
\label{rep.dens}%
\end{equation}
(with notation in the obvious correspondence). Here $M\!\left(  _{\!_{\!_{\,}%
}}\mathrm{d}(s,y)\right)  $ is the martingale measure related to (\ref{mart})
and $I\!\left(  _{\!_{\!_{\,}}}\mathrm{d}(s,y)\right)  $ the random measure
related to (\ref{incr}).

Let $\,\Delta X_{s}:=X_{s}-X_{s-\,},$\thinspace\ $s\in(0,t),$\thinspace
\ denote the jumps of the measure-valued process $X$\ by time $t.$ Recall that
they are of the form $\,r\delta_{x\,}.$\thinspace\ By an abuse of notation, we
also write $\,r=:\Delta X_{s}(x).$\thinspace\ Put
\begin{equation}
f_{s,x}\,:=\,\log\bigl((t-s)^{-1}\bigr)\,\mathsf{1}_{\{x\neq0\}}%
\log\bigl(|x|^{-1}\bigr). \label{def_f}%
\end{equation}
As a further preparation we turn to the following lemma. Recall that $t>0$ is fixed.

\begin{lemma}
[\textbf{A jump mass estimate}]\label{6.00}Fix $\,\,X_{0}=\mu\in
\mathcal{M}_{\mathrm{f}}\backslash\{0\}.$\thinspace\ Suppose\/ $\,d=1$%
\thinspace\ and $\,\alpha>1+\beta.$\thinspace\ Let \thinspace$\varepsilon
>0$\thinspace\ and $\,q>0.$\thinspace\ There exists a constant \thinspace
$c_{(\ref{in6.1a})}=c_{(\ref{in6.1a})}(\varepsilon,q)$ such that%
\begin{equation}
\mathbf{P}\Big(\Delta X_{s}(x)>c_{(\ref{in6.1a})}\bigl((t-s)|x|\bigr)^{\frac
{1}{1+\beta}}(f_{s,x})^{\ell}\text{ for some }s<t,\text{\ }\ x\in
B_{1/\mathrm{e}}(0)\Big)\!\leq\varepsilon, \label{in6.1a}%
\end{equation}
where
\begin{equation}
\ell\,:=\,\frac{1}{1+\beta}+q. \label{def.lambda1}%
\end{equation}
{}
\end{lemma}

\begin{proof}
For any $c>0$ (later to be specialized to some $c_{(\ref{in6.1a})})$ set
\[
Y\,:=\,N\Bigl((s,x,r):\;(s,x)\in(0,t)\times B_{1/\mathrm{e}}(0),\ r\geq
c\,\bigl((t-s)|x|\bigr)^{1/(1+\beta)}(f_{s,x})^{\ell}\Bigr).
\]
Clearly,%
\begin{equation}%
\begin{array}
[c]{c}%
\displaystyle
\mathbf{P}\Big(\Delta X_{s}(x)>c\,\bigl((t-s)|x|\bigr)^{1/(1+\beta)}%
(f_{s,x})^{\ell}\text{ for some }s<t\ \text{and}\ x\in B_{1/\mathrm{e}%
}(0)\Big)\vspace{6pt}\\
=\,\mathbf{P}(Y\geq1)\,\leq\,\mathbf{E}Y,
\end{array}
\label{6.20}%
\end{equation}
where in the last step we have used the classical Markov inequality.
\textrm{F}rom (\ref{decomp}),
\begin{align}
\mathbf{E}Y\,  &  =\,\varrho\,\mathbf{E}\int_{0}^{t}\mathrm{d}s\int
_{\mathsf{R}}X_{s}(\mathrm{d}x)\,\mathsf{1}_{B_{1/\mathrm{e}}(0)}%
(x)\int_{c\,\bigl((t-s)|x|\bigr)^{1/(1+\beta)}(f_{s,x})^{\ell}}^{\infty
}\mathrm{d}r\ r^{-2-\beta}\nonumber\\
&  =\,\varrho\,\frac{c^{-1-\beta}}{1+\beta}\int_{0}^{t}\mathrm{d}%
s\ (t-s)^{-1}\log^{-1-q(1+\beta)}\bigl((t-s)^{-1}\bigr)\\
&  \hspace{1cm}\times\int_{\mathsf{R}}\mathbf{E}X_{s}(\mathrm{d}%
x)\,\mathsf{1}_{B_{1/\mathrm{e}}(0)}(x)\,|x|^{-1}\log^{-1-q(1+\beta
)}\bigl(|x|^{-1}\bigr).\nonumber
\end{align}
Now, writing $C$ for a generic constant (which may change from place to
place),%
\begin{align}
&  \int_{\mathsf{R}}\,\mathbf{E}X_{s}(\mathrm{d}x)\,\mathsf{1}%
_{B_{1/\mathrm{e}}(0)}(x)\,|x|^{-1}\log^{-1-q(1+\beta)}\bigl(|x|^{-1}%
\bigr)\,\nonumber\\
&  \leq\ \mathrm{e}^{|a|t}\int_{\mathsf{R}}\mu(\mathrm{d}y)\!\int_{\mathsf{R}%
}\mathrm{d}x\ p_{s}^{\alpha}(x-y)\,\mathsf{1}_{B_{1/\mathrm{e}}(0)}%
(x)\,|x|^{-1}\log^{-1-q(1+\beta)}\bigl(|x|^{-1}\bigr)\nonumber\\
&  \leq\ C\,\mu(\mathsf{R})\,s^{-1/\alpha}\!\int_{\mathsf{R}}\mathrm{d}%
x\ \mathsf{1}_{B_{1/\mathrm{e}}(0)}(x)\,|x|^{-1}\log^{-1-q(1+\beta
)}\bigl(|x|^{-1}\bigr)\nonumber\\
&  =:\ c_{(\ref{6.3})}s^{-1/\alpha}, \label{6.3}%
\end{align}
where $\,c_{(\ref{6.3})}=c_{(\ref{6.3})}(q)$\thinspace\ (recall that $t$ is
fixed). Consequently,%
\begin{align}
\mathbf{E}Y\,  &  \leq\,\varrho\,c_{(\ref{6.3})}\,c^{-1-\beta}\int_{0}%
^{t}\mathrm{d}s\text{\thinspace}s^{-1/\alpha}\,(t-s)^{-1}\log^{-1-q(1+\beta
)}\bigl((t-s)^{-1}\bigr)\,\nonumber\\[1pt]
&  =:\,c_{(\ref{6.4})}\,c^{-1-\beta} \label{6.4}%
\end{align}
with $\,c_{(\ref{6.4})}=c_{(\ref{6.4})}(q).$\thinspace\ Choose now $c\,\ $such
that the latter expression equals $\varepsilon$ and write $c_{(\ref{in6.1a})}$
instead of $\,c.$ Recalling (\ref{6.20}), the proof is complete.
\end{proof}

Since $\,\sup_{0<y<1}y^{\gamma}\log^{\ell}\frac{1}{y}<\infty$\thinspace\ for
every $\gamma>0$, we get from Lemma~\ref{6.00} the following statement.

\begin{corollary}
[\textbf{A jump mass estimate}]\label{6.1}Fix $\,\,X_{0}=\mu\in\mathcal{M}%
_{\mathrm{f}}\backslash\{0\}.$\thinspace\ Suppose\/ $\,d=1$\thinspace\ and
$\,\alpha>1+\beta.$\thinspace\ Let \thinspace$\varepsilon>0$\thinspace\ and
$\,\gamma\in\bigl(0,(1+\beta)^{-1}\bigr).$\thinspace\ There exists a constant
\thinspace$c_{(\ref{in6.1})}=c_{(\ref{in6.1})}(\varepsilon,\gamma)$ such that
\begin{equation}
\mathbf{P}\Big(\Delta X_{s}(x)>c_{(\ref{in6.1})}%
\,\bigl((t-s)|x|\bigr)^{\lambda}\text{ for some }s<t\ \text{and}\ x\in
B_{2}(0)\Big)\leq\,\varepsilon, \label{in6.1}%
\end{equation}
where
\begin{equation}
\lambda\,:=\,\frac{1}{1+\beta}-\gamma. \label{def.lambda}%
\end{equation}
{}
\end{corollary}

Several times we will use the following estimate concerning the $\alpha
$-stable transition kernel $p^{\alpha}$ taken from \cite[Lemma~2.1]%
{FleischmannMytnikWachtel2010.optimal.Ann}.

\begin{lemma}
[$\alpha$\textbf{-stable density increment}]\label{L1old} For every $\delta
\in\lbrack0,1],$
\begin{equation}
\bigl|p_{t}^{\alpha}(x)-p_{t}^{\alpha}(y)\bigr|\,\leq\,C\,\frac{|x-y|^{\delta
}}{t^{\delta/\alpha}}\,\bigl(p_{t}^{\alpha}(x/2)+p_{t}^{\alpha}%
(y/2)\bigr),\quad t>0,\ \,x,y\in\mathsf{R}. \label{L1.1}%
\end{equation}
{}
\end{lemma}

In the proof of our main result we need also a further technical result we
quote from \cite[Lemma~2.3]{FleischmannMytnikWachtel2010.optimal.Ann}. Let
$L=\{L_{t}:\,t\geq0\}$ denote a spectrally positive stable process of index
$\kappa\in(1,2).$ Per definition, $L$ is an $\mathsf{R}$-valued
time-homogeneous process with independent increments and with Laplace
transform given by%
\begin{equation}
\mathbf{E}\,\mathrm{e}^{-\lambda L_{t}}\,=\,\mathrm{e}^{t\lambda^{\kappa}%
},\quad\lambda,t\geq0. \label{Laplace}%
\end{equation}
Note that $L$ is the unique (in law) solution to the following martingale
problem:%
\begin{equation}
t\mapsto\mathrm{e}^{-\lambda L_{t}}-\int_{0}^{t}\!\mathrm{d}s\ \mathrm{e}%
^{-\lambda L_{s}}\lambda^{\kappa}\,\ \text{is a martingale for any}%
\,\ \lambda>0. \label{MP}%
\end{equation}

Let $\,\Delta L_{s}:=L_{s}-L_{s-}>0$\thinspace\ denote the jumps of $\,L.$

\begin{lemma}
[\textbf{Big values of the process in the case of bounded jumps}]%
\label{L3}\quad\newline We have
\begin{equation}
\mathbf{P}\Bigl(\,\sup_{0\leq u\leq t}L_{u}\mathsf{1}\bigl\{\sup_{0\leq v\leq
u}\Delta L_{v}\leq y\bigr\}\geq x\Bigr)\,\leq\,\Bigl(\frac{C\,t}{xy^{\kappa
-1}}\Bigr)^{\!x/y},\quad t>0,\ \,x,y>0.
\end{equation}
{}
\end{lemma}

\section{H\"{o}lder continuity at a given point: proof of
Theorem~$\mathrm{\ref{T.prop.dens.fixed}(a)}$\label{sec:6}}

We will use some ideas from the proofs in Section~3 of
\cite{FleischmannMytnikWachtel2010.optimal.Ann}. However, to be adopted to our
case, those proofs require significant changes. Let $\,d=1$\thinspace\ and fix
$\,t,z,\mu,\alpha,\beta,\eta$\thinspace\ as in the theorem. Consider an $x\in
B_{1}(z)$. Without loss of generality we will assume that $t\leq1$ and,
changing $\mu$ appropriately, that $z=0$ and $0<x<1$. By definition
(\ref{rep.dens}) of $Z_{t}^{2}$,%
\begin{equation}
Z_{t}^{2}(x)-Z_{t}^{2}(0)\,=\,\int_{(0,t]\times\mathsf{R}}\!M\!\left(
_{\!_{\!_{\,}}}\mathrm{d}(s,y)\right)  \varphi_{+}(s,y)-\int_{(0,t]\times
\mathsf{R}}\!M\!\left(  _{\!_{\!_{\,}}}\mathrm{d}(s,y)\right)  \varphi
_{-}(s,y),
\end{equation}
where $\varphi_{+}(s,y)$ and $\varphi_{-}(s,y)$ are the positive and negative
parts of $\,p_{t-s}^{\alpha}(y-x)-p_{t-s}^{\alpha}(y)$ (for the fixed $x).$ It
is easy to check that $\varphi_{+}$ and $\varphi_{-}$ satisfy the assumptions
in \cite[Lemma~2.15]{FleischmannMytnikWachtel2010.optimal.Ann}. Thus, there
exist spectrally positive stable processes $L^{+}$ and $L^{-}$ such that
\begin{equation}
Z_{t}^{2}(x)-Z_{t}^{2}(0)\,=\,L_{T_{+}}^{+}-L_{T_{-}}^{-}, \label{6.T1.1}%
\end{equation}
where $\,T_{\pm}$\thinspace$:=$\thinspace$\int_{0}^{t}\mathrm{d}%
s\int_{\mathsf{R}}X_{s}(\mathrm{d}y)\,\bigl(\varphi_{\pm}(s,y)\bigr)^{1+\beta
}.$\thinspace\ Fix any%
\begin{equation}
\varepsilon\in\Big(0,\,\frac{\,1}{3}\Big)\quad\text{and}\quad\gamma
\in\Big(0,\,\min\Big\{\frac{\eta_{c}}{2\alpha}\,,\,\frac{\bar{\eta}_{c}%
}{2(2\alpha+1)}\Big\}\,\Big).
\end{equation}
Also fix some $J=J(\gamma)$ and
\begin{equation}
0\,=:\,\rho_{0}\,<\,\rho_{1}\,<\,\cdots\,<\,\rho_{J}\,:=\,1/\alpha
\end{equation}
such that
\begin{equation}
\rho_{\ell}\,(\alpha+1)-\frac{\rho_{\ell+1}}{1+\beta}\,\geq\,-\frac{\gamma}%
{2}\,,\quad0\leq\ell\leq J-1. \label{6.6}%
\end{equation}
According to \cite[Lemma~2.11]{FleischmannMytnikWachtel2010.optimal.Ann},
there exists a constant $c_{\varepsilon}$ such that
\begin{equation}
\mathbf{P}(V\leq c_{\varepsilon})\geq1-\varepsilon, \label{def_ve}%
\end{equation}
where%
\begin{equation}
V\,:=\,\sup_{0\leq s\leq t,\,y\in B_{2}(0)}S_{2^{\alpha}(t-s)}X_{s}\,(y)
\end{equation}
(note that there is no difference in using $B_{2}(0)$ or its closure for
taking the supremum). By Lemma~\ref{6.1} we can fix $c_{(\ref{in6.1})}$
sufficiently large such that the probability of the event
\begin{equation}
A^{\varepsilon,1}\,:=\,\Big\{\Delta X_{s}(y)\leq c_{(\ref{in6.1}%
)}\,\bigl((t-s)|y|\bigr)^{\lambda}\text{ for all }s<t\ \text{and}\ y\in
B_{2}(0)\Big\} \label{6.A.eps}%
\end{equation}
is larger than $1-\varepsilon$. Moreover, according to \cite[Lemma~2.14]%
{FleischmannMytnikWachtel2010.optimal.Ann}, there exists a constant $c^{\ast
}=c^{\ast}(\varepsilon,\gamma)$ such that the probability of the event
\begin{equation}
A^{\varepsilon,2}\,:=\,\Big\{\Delta X_{s}(y)\leq c^{\ast}(t-s)^{\lambda}\text{
for all }s<t\ \text{and }y\in\mathsf{R}\Big\} \label{6.A.eps2}%
\end{equation}
is larger than $1-\varepsilon$. Set
\begin{equation}
A^{\varepsilon}\,:=\,\,A^{\varepsilon,1}\cap A^{\varepsilon,2}\cap\{V\leq
c_{\varepsilon}\}. \label{6.A1.eps}%
\end{equation}
Evidently,
\begin{equation}
\mathbf{P}(A^{\varepsilon})\geq1-3\varepsilon. \label{6.A.eps.est}%
\end{equation}

Define $Z_{t}^{2,\varepsilon}:=Z_{t}^{2}\,\mathsf{1}(A^{\varepsilon}%
).$\thinspace\ We first show that $Z_{t}^{2,\varepsilon}$ has a version which
is locally H\"{o}lder continuous of all orders $\eta$ less than $\bar{\eta
}_{\mathrm{c}\,}.$\thinspace\ It follows from (\ref{6.T1.1}) that,\textbf{
}for any $k>0$,
\begin{align}
&  \mathbf{P}\Big(\bigl|Z_{t}^{2,\varepsilon}(x)-Z_{t}^{2,\varepsilon
}(0)\bigr|\geq2k\,x^{\eta}\Big)\nonumber\\
\,  &  \leq\,\mathbf{P}\bigl(L_{T_{+}}^{+}\geq kx^{\eta},\,A^{\varepsilon
}\bigr)+\mathbf{P}\bigl(L_{T_{-}}^{-}\geq kx^{\eta},\,A^{\varepsilon}\bigr).
\label{6.T1.2}%
\end{align}
Define:
\begin{equation}
\tilde{D}_{0}\,:=\,\Big\{(s,y)\in\lbrack0,t)\times B_{2}(0):\;y\in
\bigl(-2(t-s)^{1/\alpha-\rho_{1}},\,x+2(t-s)^{1/\alpha-\rho_{1}}\bigr)\Big\}
\end{equation}
and, for $\,1\leq\ell\leq J-1,$%
\[
\tilde{D}_{\ell}\,:=\,\Big\{(s,y)\in\lbrack0,t)\times B_{2}(0):\;y\in
\bigl(-2(t-s)^{1/\alpha-\rho_{\ell+1}},\,x+2(t-s)^{1/\alpha-\rho_{\ell+1}%
}\bigr)\Big\}.
\]
Moreover,%
\begin{equation}
D_{0}:=\tilde{D}_{0}\quad\text{and}\quad D_{\ell}\,:=\,\tilde{D}_{\ell
}\setminus\tilde{D}_{\ell-1\,},\quad1\leq\ell\leq J-1.
\end{equation}
Note that
\begin{equation}
\lbrack0,t)\times B_{2}(0)\,=\,\bigcup_{0\leq\ell<J}D_{\ell\,}.
\end{equation}
If the jumps of \thinspace$M\!\left(  _{\!_{\!_{\,}}}\mathrm{d}(s,y)\right)  $
do not exceed $c_{(\ref{in6.1})}\,\bigl((t-s)|y|\bigr)^{\lambda}$ on
$D_{\ell\,},$\thinspace\ then the jumps of the process $\,u\mapsto$
$\int_{(0,u]\times D_{\ell}}\!M\!\left(  _{\!_{\!_{\,}}}\mathrm{d}%
(s,y)\right)  \varphi_{\pm}(s,y)$ are bounded by%
\begin{equation}
c_{(\ref{in6.1})}\,\sup_{(s,y)\in D_{\ell}}\bigl((t-s)|y|\bigr)^{\lambda
}\,\varphi_{\pm}(s,y). \label{6.T1.3}%
\end{equation}
For $0\leq\ell<J,$ put
\begin{align}
D_{\ell,1}\,  &  :=\,\bigl\{(s,y)\in D_{\ell}:\;(t-s)^{1/\alpha-\rho_{\ell+1}%
}\leq x\bigr\},\\
D_{\ell,2}\,  &  :=\,\bigl\{(s,y)\in D_{\ell}:\;(t-s)^{1/\alpha-\rho_{\ell+1}%
}>x\bigr\},\nonumber\\
D_{\ell,1}(s)\,  &  :=\,\left\{  _{\!_{\!_{\,}}}y\in B_{2}(0):(s,y)\in
D_{\ell,1}\right\}  ,\quad s\in\lbrack0,t),\nonumber\\
D_{\ell,2}(s)\,  &  :=\,\left\{  _{\!_{\!_{\,}}}y\in B_{2}(0):(s,y)\in
D_{\ell,2}\right\}  ,\quad s\in\lbrack0,t).\nonumber
\end{align}
Since obviously $D_{\ell}=D_{\ell,1}\cup D_{\ell,2}$ we get that~(\ref{6.T1.3}%
) is bounded by
\begin{align}
\lefteqn{c_{(\ref{in6.1})}\,\sup_{0<s<t}(t-s)^{\lambda}\sup_{y\in D_{\ell
,1}(s)}|y|^{\lambda}\,\varphi_{\pm}(s,y)}\nonumber\\
&  \qquad+c_{(\ref{in6.1})}\,\sup_{0<s<t}(t-s)^{\lambda}\sup_{y\in D_{\ell
,2}(s)}|y|^{\lambda}\,\varphi_{\pm}(s,y)\,=:\,c_{(\ref{in6.1})}(I_{1}+I_{2}).
\label{6.T1.10}%
\end{align}
Clearly,
\begin{equation}
\varphi_{\pm}(s,y)\,\leq\,\big|p_{t-s}^{\alpha}(y-x)-p_{t-s}^{\alpha
}(y)\big|,\quad\text{for all }s,y. \label{6.T1.4}%
\end{equation}

First let us bound $I_{1\,}.$\thinspace\ Note that for any $(s,y)\in
D_{\ell,1\,},$
\begin{equation}
|y|\,\leq\,x+2\,(t-s)^{1/\alpha-\rho_{\ell+1}}\,\leq\,3\,x.
\end{equation}
Therefore we have
\begin{equation}
I_{1}\,\leq\,3^{\lambda}\,x^{\lambda}\,\sup_{0<s<t}(t-s)^{\lambda}\sup_{y\in
D_{\ell,1}(s)}\,\big|p_{t-s}^{\alpha}(y-x)-p_{t-s}^{\alpha}(y)\big|.
\label{6.10}%
\end{equation}
Using Lemma~\ref{L1old} with $\delta=\eta_{\mathrm{c}}-2\alpha\gamma$ gives
\begin{align}
&  \sup_{y\in D_{\ell,1}(s)}\big|p_{t-s}^{\alpha}(y-x)-p_{t-s}^{\alpha
}(y)\big|\,\nonumber\\
&  \leq\,C\,x^{\eta_{\mathrm{c}}-2\alpha\gamma}\,(t-s)^{-\eta_{\mathrm{c}%
}/\alpha+2\gamma}\,\nonumber\\
&  \qquad\times\,\sup_{y\in D_{\ell,1}(s)}\Big(p_{t-s}^{\alpha}\left(
_{\!_{\!_{\,}}}(y-x)/2\right)  +p_{t-s}^{\alpha}\left(  y/2\right)
\!\Big)\nonumber\\
&  =\,C\,x^{\eta_{\mathrm{c}}-2\alpha\gamma}\,(t-s)^{-\eta_{\mathrm{c}}%
/\alpha+2\gamma-1/\alpha}\,\nonumber\\
&  \qquad\times\,\sup_{y\in D_{\ell,1}(s)}\Big(p_{1}^{\alpha}%
\bigl((t-s)^{-1/\alpha}(y-x)/2\bigr)+p_{1}^{\alpha}\bigl((t-s)^{-1/\alpha
}y/2\bigr)\Big). \label{star}%
\end{align}
Recall the following standard estimate on $p_{1}^{\alpha}$:
\begin{equation}
p_{1}^{\alpha}(y)\,\leq\,c_{(\ref{tail})}\,|y|^{-(\alpha+1)},\quad
y\in\mathsf{R}, \label{tail}%
\end{equation}
for some constant $c_{(\ref{tail})}.$ Thus on $D_{\ell,1}(s),$ we have
$\,|y|\geq2(t-s)^{1/\alpha-\rho_{\ell}},$\thinspace\ implying%
\begin{equation}
p_{1}^{\alpha}\bigl((t-s)^{-1/\alpha}y/2\bigr)\leq p_{1}^{\alpha
}\bigl((t-s)^{-\rho_{\ell}}\bigr)\leq c_{(\ref{tail})}\,(t-s)^{\rho_{\ell
}(\alpha+1)},
\end{equation}
where the last inequality follows by (\ref{tail}). A similar estimate holds
for the second $p_{1}^{\alpha}$-expression in (\ref{star}). Thus, (\ref{star})
yields%
\begin{equation}
\sup_{y\in D_{\ell,1}(s)}\big|p_{t-s}^{\alpha}(y-x)-p_{t-s}^{\alpha
}(y)\big|\leq\,C\,x^{\eta_{\mathrm{c}}-2\alpha\gamma}\,(t-s)^{-\eta
_{\mathrm{c}}/\alpha+2\gamma-1/\alpha+\rho_{\ell}(\alpha+1)}. \label{yields}%
\end{equation}
Now let us check that
\begin{equation}
\sup_{0<s<t}\,(t-s)^{\lambda}\,(t-s)^{-\eta_{\mathrm{c}}/\alpha+2\gamma
-1/\alpha+\rho_{\ell}(\alpha+1)}\,\leq\,1. \label{6.7}%
\end{equation}
Recall that $\eta_{\mathrm{c}}=\frac{\alpha}{1+\beta}-1$. Then one can easily
get that
\begin{equation}
\lambda-\eta_{\mathrm{c}}/\alpha+2\gamma-1/\alpha+\rho_{\ell}(\alpha
+1)\,=\,\gamma+\rho_{\ell}(\alpha+1)\,\geq\,\gamma, \label{6.8}%
\end{equation}
where the last inequality follows by~(\ref{6.6}). Therefore~(\ref{6.7})
follows immediately. Combining (\ref{6.10}), (\ref{yields}), and (\ref{6.7}),
we see that
\begin{equation}
I_{1}\,\leq\,C\,x^{\lambda+\eta_{\mathrm{c}}-2\alpha\gamma}\,\leq
\,C\,x^{\bar{\eta}_{\mathrm{c}}-(2\alpha+1)\gamma}, \label{6.16}%
\end{equation}
where we used the definitions of $\eta_{\mathrm{c}}$ and $\bar{\eta
}_{\mathrm{c}\,},$\thinspace\ given in (\ref{eta_c}) and
Theorem~\ref{T.prop.dens.fixed}(a), respectively.

Now let us bound $I_{2\,}.$\thinspace\ Note that for any $(s,y)\in
D_{\ell,2\,},$
\begin{equation}
|y|\,\leq\,x+2\,(t-s)^{1/\alpha-\rho_{\ell+1}}\,\leq\,3\,(t-s)^{1/\alpha
-\rho_{\ell+1}}.
\end{equation}
Therefore we have
\begin{equation}
I_{2}\,\leq\,3^{\lambda}\,\sup_{0<s<t}\!\Big((t-s)^{\lambda+(1/\alpha
-\rho_{\ell+1})\lambda}\,\sup_{y\in D_{\ell,2}(s)}\,\big|p_{t-s}^{\alpha
}(y-x)-p_{t-s}^{\alpha}(y)\big|\Big). \label{6.15}%
\end{equation}
Using again Lemma~\ref{L1old} but this time with $\delta=\bar{\eta
}_{\mathrm{c}}-(2\alpha+1)\gamma$ gives
\begin{align}
&  \sup_{y\in D_{\ell,2}(s)}\big|p_{t-s}^{\alpha}(y-x)-p_{t-s}^{\alpha
}(y)\big|\,\nonumber\\
&  \leq\,C\,x^{\bar{\eta}_{\mathrm{c}}-(2\alpha+1)\gamma}\,(t-s)^{-\bar{\eta
}_{\mathrm{c}}/\alpha+2\gamma+\gamma/\alpha}\,\nonumber\\[2pt]
&  \qquad\qquad\times\sup_{y\in D_{\ell,2}(s)}\Big(p_{t-s}^{\alpha}\left(
_{\!_{\!_{\,}}}(y-x)/2\right)  +p_{t-s}^{\alpha}\left(  y/2\right)
\!\Big)\nonumber\\
\,  &  =\,C\,x^{\bar{\eta}_{\mathrm{c}}-(2\alpha+1)\gamma}\,(t-s)^{-\bar{\eta
}_{\mathrm{c}}/\alpha+2\gamma+\gamma/\alpha-1/\alpha+\rho_{\ell}(\alpha+1)}.
\label{6.11}%
\end{align}
By definition (\ref{def.lambda}) of $\lambda,$
\begin{align}
&  \lambda+\Big(\frac{1}{\alpha}-\rho_{\ell+1}\Big)\lambda-\frac{\bar{\eta
}_{\mathrm{c}}}{\alpha}+2\gamma+\frac{\gamma}{\alpha}-\frac{1}{\alpha}%
+\rho_{\ell}(\alpha+1)\nonumber\\
&  =\,\frac{1}{\alpha}\Bigl(\frac{1+\alpha}{1+\beta}-1-\bar{\eta}_{\mathrm{c}%
}\Bigr)+\gamma+\gamma\rho_{\ell+1}-\frac{\rho_{\ell+1}}{1+\beta}+\rho_{\ell
}(\alpha+1)\nonumber\\
&  \geq\,\gamma/2 \label{6.13}%
\end{align}
where in the last step we used the definition of\thinspace\ $\bar{\eta
}_{\mathrm{c}}$ given in Theorem~\ref{T.prop.dens.fixed}(a), and~(\ref{6.6}).
Thus
\begin{equation}
\sup_{0<s<t}\,(t-s)^{\lambda+(1/\alpha-\rho_{\ell+1})\lambda-\bar{\eta
}_{\mathrm{c}}/\alpha+2\gamma+\gamma/\alpha-1/\alpha+\rho_{\ell}(\alpha
+1)}\,\leq\,1. \label{6.12}%
\end{equation}
Combining estimates (\ref{6.15}), (\ref{6.11}), and (\ref{6.12}), we obtain
\begin{equation}
I_{2}\,\leq\,C\,x^{\bar{\eta}_{\mathrm{c}}-(2\alpha+1)\gamma}. \label{6.14}%
\end{equation}

If the jumps of \thinspace$M\!\left(  _{\!_{\!_{\,}}}\mathrm{d}(s,y)\right)  $
are smaller than $c^{\ast}(t-s)^{\lambda}$ on $\mathsf{R}\setminus B_{2}(0)$
(where $c^{\ast}$ is from (\ref{6.A.eps2})), then the jumps of the process
$\,u\mapsto$ $\int_{(0,u]\times(\mathsf{R}\setminus B_{2}(0))}\!M\!\left(
_{\!_{\!_{\,}}}\mathrm{d}(s,y)\right)  \varphi_{\pm}(s,y)$ are bounded by
\begin{equation}
c^{\ast}(t-s)^{\lambda}\sup_{y\in\mathsf{R}\setminus B_{2}(0)}\varphi_{\pm
}(s,y). \label{V1}%
\end{equation}
Using Lemma~\ref{L1old} once again but this time with $\delta=\bar{\eta
}_{\mathrm{c}}-2\alpha\gamma$, we have
\begin{align}
\big|p_{t-s}^{\alpha}(y-x)-p_{t-s}^{\alpha}(y)\big|\,  &  \leq\,C\,x^{\bar
{\eta}_{\mathrm{c}}-2\alpha\gamma}\,(t-s)^{-\bar{\eta}_{\mathrm{c}}%
/\alpha+2\gamma}\nonumber\\
&  \quad\quad\times\Big(p_{t-s}^{\alpha}\left(  _{\!_{\!_{\,}}}(y-x)/2\right)
+p_{t-s}^{\alpha}\left(  y/2\right)  \!\Big). \label{V2}%
\end{align}
Since $0<x<1$,
\begin{align}
&  \sup_{y\in\mathsf{R}\setminus B_{2}(0)}\Big(p_{t-s}^{\alpha}\left(
_{\!_{\!_{\,}}}(y-x)/2\right)  +p_{t-s}^{\alpha}\left(  y/2\right)
\!\Big)\,\nonumber\\
&  \leq\,C\,(t-s)^{-1/\alpha}\,p_{1}^{\alpha}\bigl((t-s)^{-1/\alpha
}/2\bigr)\leq\,C\,(t-s). \label{V3}%
\end{align}
Therefore, (\ref{6.T1.4}), (\ref{V2}), and (\ref{V3}) imply
\begin{gather}
c^{\ast}(t-s)^{\lambda}\sup_{y\in\mathsf{R}\setminus B_{2}(0)}\varphi_{\pm
}(s,y)\leq C\,x^{\bar{\eta}_{\mathrm{c}}-2\alpha\gamma}\,(t-s)^{\lambda
-\bar{\eta}_{\mathrm{c}}/\alpha+2\gamma+1}\nonumber\\
\leq c_{(\ref{V4})}\,x^{\bar{\eta}_{\mathrm{c}}-2\alpha\gamma} \label{V4}%
\end{gather}
for some constant $c_{(\ref{V4})}=c_{(\ref{V4})}(\varepsilon).$ Here we have
used that $\bar{\eta}_{\mathrm{c}}\leq(1+\alpha)/(1+\beta)-1$ implies
$\lambda-\bar{\eta}_{\mathrm{c}}/\alpha+2\gamma+1\geq1$.

Combining (\ref{6.T1.3}), (\ref{6.T1.10}),~(\ref{6.16}), (\ref{6.14}), and
(\ref{V4}), we see that all jumps of the process $\,u\mapsto\int
_{(0,u]\times\mathsf{R}}\!M\!\left(  _{\!_{\!_{\,}}}\mathrm{d}(s,y)\right)
\varphi_{\pm}(s,y)$ on the set $A^{\varepsilon}$ are bounded by
\begin{equation}
c_{(\ref{C1})}\,x^{\bar{\eta}_{\mathrm{c}}-(2\alpha+1)\gamma} \label{C1}%
\end{equation}
for some constant $c_{(\ref{C1})}=c_{(\ref{C1})}(\varepsilon).$ Therefore, by
an abuse of notation writing $L$ for $L^{+}$ and $L^{-},$%
\begin{align}
&  \mathbf{P}\bigl(L_{T_{\pm}}\geq kx^{\eta},\,A^{\varepsilon}\bigr)\,\\[4pt]
&  =\,\mathbf{P}\Bigl(L_{T_{\pm}}\geq kx^{\eta},\ \sup_{0<u<T_{\pm}}\Delta
L_{u}\leq c_{(\ref{C1})}\,x^{\bar{\eta}_{\mathrm{c}}-(2\alpha+1)\gamma
},\,A^{\varepsilon}\Bigr)\nonumber\\
\,  &  \leq\,\mathbf{P}\biggl(\,\sup_{0<v\leq T_{\pm}}L_{v}\,\mathsf{1}%
\Big\{\sup_{0<u<v}\Delta L_{u}\leq c_{(\ref{C1})}\,x^{\bar{\eta}_{\mathrm{c}%
}-(2\alpha+1)\gamma}\Big\}\geq kx^{\eta},\,A^{\varepsilon}\biggr).\nonumber
\end{align}
Since
\begin{equation}
T_{\pm}\,\leq\,\int_{0}^{t}\mathrm{d}s\int_{\mathsf{R}}X_{s}(\mathrm{d}%
y)\,\big|p_{t-s}^{\alpha}(y-x)-p_{t-s}^{\alpha}(y)\big|^{1+\beta},
\end{equation}
applying \cite[Lemma~2.12]{FleischmannMytnikWachtel2010.optimal.Ann} with
\begin{equation}
\theta=1+\beta\quad\text{and}\quad\delta=\mathsf{1}_{\beta<(\alpha-1)/2}%
+\frac{\alpha-\beta-\varepsilon}{1+\beta}\ \mathsf{1}_{\beta\geq
(\alpha-1)/2\,},
\end{equation}
we may fix $\varepsilon_{1}\in(0,\,\alpha\gamma\beta)$ to get the bound
\begin{equation}
T_{\pm}\,\leq\,c_{(\ref{81_1})}\,\Big(x^{1+\beta}\,\mathsf{1}_{\beta
<(\alpha-1)/2}+x^{\alpha-\beta-\varepsilon_{1}}\,\mathsf{1}_{\beta\geq
(\alpha-1)/2}\Big) \label{81_1}%
\end{equation}
on $\{V\leq c_{\varepsilon}\}$ for some constant $c_{(\ref{81_1}%
)}=c_{(\ref{81_1})}(\varepsilon).$ Consequently,
\begin{align}
&  \mathbf{P}\bigl(L_{T_{\pm}}\geq kx^{\eta},\,A^{\varepsilon}%
\bigr)\,\,\label{A}\\
&  \leq\,\mathbf{P}\biggl(\,\sup_{0<v\leq c_{(\ref{81_1})}\left(  x^{1+\beta
}1_{\beta<(\alpha-1)/2}+x^{\alpha-\beta-\varepsilon_{1}}1_{\beta\geq
(\alpha-1)/2}\right)  }\,L_{v}\nonumber\\
&  \qquad\qquad\quad\quad\times\,\mathsf{1}\Big\{\sup_{0<u<v}\Delta L_{u}\leq
c_{(\ref{C1})}\,x^{\bar{\eta}_{\mathrm{c}}-(2\alpha+1)\gamma}\Big\}\geq
kx^{\eta}\biggr).\nonumber
\end{align}
Now use Lemma~\ref{L3} with $\,\kappa=1+\beta,$\thinspace\ $t=c_{(\ref{81_1}%
)}\Big(x^{1+\beta}1_{\beta<(\alpha-1)/2}+x_{\beta\geq(\alpha-1)/2}%
^{\alpha-\beta-\varepsilon_{1}}$ $\mathsf{1}\Big),$\thinspace\ $kx^{\eta}%
$\thinspace\ instead of $x,$\thinspace\ and $\,y=c_{(\ref{C1})}$%
\thinspace$x^{\bar{\eta}_{\mathrm{c}}-(2\alpha+1)\gamma}.$\thinspace\ This
gives
\begin{align}
&  \mathbf{P}\bigl(L_{T_{\pm}}\geq kx^{\eta},\,A^{\varepsilon}%
\bigr)\label{star2}\\
&  \leq\!\Bigg(\frac{Cc_{(\ref{81_1})}\Big(x^{1+\beta}1_{\beta<(\alpha
-1)/2}+x^{\alpha-\beta-\varepsilon_{1}}1_{\beta\geq(\alpha-1)/2}%
\Big)}{kx^{\eta}(c_{(\ref{C1})}x^{\bar{\eta}_{\mathrm{c}}-(2\alpha+1)\gamma
})^{\beta}}\Bigg)^{\!\!\frac{x^{\eta-\bar{\eta}_{\mathrm{c}}+(2\alpha
+1)\gamma}}{c_{(\ref{C1})}}}.\nonumber
\end{align}
Now we need additionally the following simple inequalities, which are easy to
derive:
\begin{equation}
-\eta-\beta\bigl(\bar{\eta}_{\mathrm{c}}-(2\alpha+1)\gamma\bigr)+1+\beta
\,\geq\,(2\alpha+1)\gamma\beta\quad\text{on}\,\ \beta<\frac{\alpha-1}{2}\,,
\end{equation}
and
\begin{equation}
-\eta-\beta\bigl(\bar{\eta}_{\mathrm{c}}-(2\alpha+1)\gamma\bigr)+\alpha
-\beta-\varepsilon_{1}\geq(2\alpha+1)\gamma\beta-\varepsilon_{1}\geq
\alpha\gamma\beta\quad\text{on}\,\ \beta\geq\frac{\alpha-1}{2}\,.
\end{equation}
In fact, $\bar{\eta}_{\mathrm{c}}=1$ under $\beta<(\alpha-1)/2,$ whereas the
other case in the definition of $\bar{\eta}_{\mathrm{c}}$ applies under
$\beta\geq(\alpha-1)/2.$ Then, using the above inequalities and (\ref{star2}),
we obtain
\begin{equation}
\mathbf{P}\bigl(L_{T_{\pm}}\geq kx^{\eta},\,A^{\varepsilon}\bigr)\,\leq
\,\bigl(c_{(\ref{48})}\,k^{-1}x^{\alpha\gamma\beta}\bigr)^{\bigl(c_{(\ref{C1}%
)}^{-1}kx^{\eta-\bar{\eta}_{\mathrm{c}}+(2\alpha+1)\gamma}\bigr)} \label{48}%
\end{equation}
for some constant $c_{(\ref{48})}=c_{(\ref{48})}(\varepsilon).$ Applying this
bound with $\gamma=\frac{\bar{\eta}_{\mathrm{c}}-\eta}{2(2\alpha+1)}$ to the
summands at the right hand side in inequality (\ref{6.T1.2}), and noting that
$\,\alpha\gamma\beta$\thinspace\ is also a positive constant here, we have
\begin{equation}
\mathbf{P}\Big(\bigl|Z_{t}^{2,\varepsilon}(x)-Z_{t}^{2,\varepsilon
}(0)\bigr|\geq2k\,x^{\eta}\Big)\,\leq\,2\bigl(c_{(\ref{48})}k^{-1}%
x\bigr)^{\left(  c_{_{(\ref{49})}}kx^{(\eta-\bar{\eta}_{\mathrm{c}}%
)/2}\right)  } \label{49}%
\end{equation}
for some constant $c_{_{(\ref{49})}}.$ This inequality yields%
\begin{equation}
\lim_{k\rightarrow\infty}\sum_{n=1}^{\infty}\mathbf{P}\Big(\bigl|Z_{t}%
^{2,\varepsilon}(n^{-q})-Z_{t}^{2,\varepsilon}(0)\bigr|\geq kn^{-q\eta}\Big)=0
\label{50}%
\end{equation}
for every positive $q$.

Recall that our purpose is to show that
\begin{equation}
\sup_{0<x<1}\frac{\bigl|Z_{t}^{2}(x)-Z_{t}^{2}(0)\bigr|}{x^{\eta}}<\infty
\quad\text{almost surely,}%
\end{equation}
or, in other words,
\begin{equation}
\lim_{k\uparrow\infty}\mathbf{P}\biggl(\,\sup_{0<x<1}\frac{\bigl|Z_{t}%
^{2}(x)-Z_{t}^{2}(0)\bigr|}{x^{\eta}}>k\biggr)\,=\,0. \label{6.T1.6}%
\end{equation}
It is easy to see that
\begin{equation}
\left\{  \sup_{0<x<1}\frac{\bigl|Z_{t}^{2}(x)-Z_{t}^{2}(0)\bigr|}{x^{\eta}%
}>k\right\}  \subseteq\,\bigcup_{n=1}^{\infty}\left\{  \sup_{x\in I_{n}%
}\bigl|Z_{t}^{2}(x)-Z_{t}^{2}(0)\bigr|>\frac{k}{2^{q}}n^{-q\eta}\right\}  \!,
\end{equation}
where $I_{n}:=\{x:(n+1)^{-q}\leq x<n^{-q}\}$. Moreover, by the triangle
inequality,
\begin{equation}
\bigl|Z_{t}^{2}(x)-Z_{t}^{2}(0)\bigr|\leq\bigl|Z_{t}^{2}(x)-Z_{t}^{2}%
(n^{-q})\bigr|+\bigl|Z_{t}^{2}(n^{-q})-Z_{t}^{2}(0)\bigr|,\quad x\in I_{n\,}.
\end{equation}
Furthermore, for all $R>0$,
\begin{align}
&  \left\{  \sup_{0<x<y<1}\frac{\bigl|Z_{t}^{2}(x)-Z_{t}^{2}(y)\bigr|}%
{|x-y|^{q\eta/(q+1)}}\leq R\right\} \nonumber\\
&  \subseteq\left\{  \bigl|Z_{t}^{2}(x)-Z_{t}^{2}(n^{-q})\bigr|\leq
R\,q^{q\eta/(q+1)}n^{-q\eta},\,\ x\in I_{n}\right\}  \!.
\end{align}
Consequently, for all $n\geq1$,
\begin{align*}
&  \left\{  \sup_{x\in I_{n}}\bigl|Z_{t}^{2}(x)-Z_{t}^{2}(0)\bigr|>\frac
{k}{2^{q}}n^{-q\eta}\right\} \\
&  \subseteq\left\{  \sup_{0<x<y<1}\frac{\bigl|Z_{t}^{2}(x)-Z_{t}%
^{2}(y)\bigr|}{|x-y|^{q\eta/(q+1)}}>c(q)k\right\}  \cup\left\{  \bigl|Z_{t}%
^{2}(n^{-q})-Z_{t}^{2}(0)\bigr|>\frac{k}{2^{q+1}}n^{-q\eta}\right\}  \!,
\end{align*}
where $c(q)$ is some positive constant. If we choose $q$ so small that
\thinspace$\eta q/(q+1)<\eta_{c\,},$\thinspace\ then
\begin{equation}
\lim_{k\rightarrow\infty}\mathbf{P}\biggl(\sup_{0<x<y<1}\frac{\bigl|Z_{t}%
^{2}(x)-Z_{t}^{2}(y)\bigr|}{|x-y|^{q\eta/(q+1)}}>c(q)k\biggr)=0,
\end{equation}
since, by Theorem~1.2(a) of \cite{FleischmannMytnikWachtel2010.optimal.Ann},
$Z_{t}^{2}$ is locally H\"{o}lder continuous of every index smaller than
$\eta_{c\,}.$ Therefore, it suffices to show that
\begin{equation}
\lim_{k\rightarrow\infty}\mathbf{P}\biggl(\bigcup_{n=1}^{\infty}\left\{
\bigl|Z_{t}^{2}(n^{-q})-Z_{t}^{2}(0)\bigr|>\frac{k}{2^{q+1}}n^{-q\eta
}\right\}  \biggr)=0.
\end{equation}
But
\begin{align}
&  \mathbf{P}\biggl(\bigcup_{n=1}^{\infty}\left\{  \bigl|Z_{t}^{2}%
(n^{-q})-Z_{t}^{2}(0)\bigr|>\frac{k}{2^{q+1}}n^{-q\eta}\right\}  \biggr)\\
&  \leq\,\mathbf{P}\biggl(\bigcup_{n=1}^{\infty}\left\{  \bigl|Z_{t}%
^{2,\varepsilon}(n^{-q})-Z_{t}^{2,\varepsilon}(0)\bigr|>\frac{k}{2^{q+1}%
}n^{-q\eta}\right\}  \biggr)+\mathbf{P}(A^{\varepsilon,\mathrm{c}}),\nonumber
\end{align}
where $A^{\varepsilon,\mathrm{c}}$ denotes the complement of $A^{\varepsilon}%
$. It follows from (\ref{50}) that
\begin{equation}
\lim_{k\rightarrow\infty}\mathbf{P}\biggl(\bigcup_{n=1}^{\infty}\left\{
\bigl|Z_{t}^{2,\varepsilon}(n^{-q})-Z_{t}^{2,\varepsilon}(0)\bigr|>\frac
{k}{2^{q+1}}n^{-q\eta}\right\}  \biggr)=0.
\end{equation}
Moreover, $\mathbf{P}(A^{\varepsilon,\mathrm{c}})\leq3\varepsilon$, see
(\ref{6.A.eps.est}). As a result we have
\begin{equation}
\limsup_{k\uparrow\infty}\mathbf{P}\biggl(\bigcup_{n=1}^{\infty}\left\{
\bigl|Z_{t}^{2}(n^{-q})-Z_{t}^{2}(0)\bigr|>\frac{k}{2^{q+1}}n^{-q\eta
}\right\}  \biggr)\,\leq\,3\varepsilon.
\end{equation}
Since $\varepsilon$ may be arbitrarily small, this implies (\ref{6.T1.6}).
This yields the desired H\"{o}lder continuity of $\,Z_{t}^{2}$\thinspace\ at
$0,$ for all \thinspace$\eta<\bar{\eta}_{\mathrm{c}\,}.$\thinspace\ Since
$Z_{t}^{1}$ and $Z_{t}^{3}$ are a.s. Lipschitz continuous at $0$
(cf.\ \cite[Remark~2.13]{FleischmannMytnikWachtel2010.optimal.Ann}), recalling
(\ref{rep.dens}), the proof of Theorem~\ref{T.prop.dens.fixed}(a) is
complete.\hfill$\square$\bigskip

\section{Optimality of $\bar{\eta}_{\mathrm{c}}\!:$ proof of
Theorem~$\mathrm{\ref{T.prop.dens.fixed}(b)}$\label{S.opt}}

We continue to consider $\,d=1,$\thinspace\ to fix $\,t,z,\mu,\alpha
,\beta,\eta$\thinspace\ as in the theorem, and to assume $0<t<1$ and $z=0.$

In analogy to the proof of optimality of $\eta_{c}$ in \cite[Section~5]%
{FleischmannMytnikWachtel2010.optimal.Ann}, our strategy is to find a sequence
of \textquotedblleft big\textquotedblright\ jumps that occur close to time
$t$. But in contrast to the case of the local H\"{o}lder continuity, we need
to find these \textquotedblleft big\textquotedblright\ jumps in the vicinity
of 0, where these jumps should destroy the H\"{o}lder continuity of any index
greater or equal than $\bar{\eta}_{\mathrm{c}\,}.$\thinspace\ This needs to
overcome some new technical difficulties.

Recall that we need to prove the optimality in the case $\beta>(\alpha-1)/2$
only. This implies that $\bar{\eta}_{\mathrm{c}}=\frac{\alpha+1}{\beta+1}-1<1$.

First let us give two technical lemmas that we need for the proof.

\begin{lemma}
[\textbf{Some left-hand continuity}]\label{n.L1} For all \thinspace
$c,\theta>0$,
\begin{equation}
\mathbf{P}\!\left(  X_{t}(0)>\theta,\,\ \liminf_{s\uparrow t}S_{t-s}^{\alpha
}X_{s}\bigl(c\,(t-s)^{1/\alpha}\bigr)\leq\theta\right)  =\,0.
\end{equation}
{}
\end{lemma}

\begin{proof}
For brevity, set
\begin{equation}
A\,:=\left\{  \liminf_{s\uparrow t}S_{t-s}^{\alpha}X_{s}%
\bigl(c\,(t-s)^{1/\alpha}\bigr)\leq\theta\right\}  ,
\end{equation}
and for $n>1/t$ define the stopping times
\begin{equation}
\tau_{n}\,:=\,\left\{  \!%
\begin{array}
[c]{l}%
\inf\Big\{s\in(t-1/n,t):S_{t-s}^{\alpha}X_{s}\bigl(c\,(t-s)^{1/\alpha
}\bigr)\leq\theta+1/n\Big\},\quad\omega\in A,\vspace{4pt}\\
t,\quad\omega\in A^{\mathrm{c}}.
\end{array}
\right.
\end{equation}
Define also
\begin{equation}
x_{n}:=\,c\,(t-\tau_{n})^{1/\alpha}.
\end{equation}
Then, using the strong Markov property, we get
\begin{equation}
\mathbf{E}\!\left[  X_{t}(x_{n})\,\big|\,\mathcal{F}_{\tau_{n}}\right]
=\,S_{t-\tau_{n}}^{\alpha}X_{\tau_{n}}(x_{n})=\,X_{t}(0)\mathsf{1}%
_{A^{\mathrm{c}}}+S_{t-\tau_{n}}^{\alpha}X_{\tau_{n}}(x_{n})\mathsf{1}_{A\,}.
\label{n1}%
\end{equation}
We next note that $x_{n}\rightarrow0$ almost surely as $n\uparrow\infty$. This
implies, in view of the continuity of $X_{t}$ at zero, that $X_{t}%
(x_{n})\rightarrow X_{t}(0)$ almost surely. Recalling that \thinspace
$\mathbf{E}\sup_{|x|\leq1}X_{t}(x)<\infty$\thinspace\ in view of Corollary~2.8
of \cite{FleischmannMytnikWachtel2010.optimal.Ann}, we conclude that
\begin{equation}
X_{t}(x_{n})\,\underset{n\uparrow\infty}{\longrightarrow}\,X_{t}%
(0)\quad\text{in }\mathcal{L}_{1\,}.
\end{equation}
This, in its turn, implies that
\begin{equation}
\mathbf{E}\left[  X_{t}(x_{n})|\mathcal{F}_{\tau_{n}}\right]  -\mathbf{E}%
\left[  X_{t}(0)|\mathcal{F}_{\tau_{n}}\right]  \,\underset{n\uparrow\infty
}{\longrightarrow}\,0\quad\text{in }\mathcal{L}_{1\,}. \label{n2}%
\end{equation}
Furthermore, it follows from the well known Levy theorem on convergence of
conditional expectations that
\begin{equation}
\mathbf{E}\!\left[  X_{t}(0)\,\big|\,\mathcal{F}_{\tau_{n}}\right]
\,\underset{n\uparrow\infty}{\longrightarrow}\,\mathbf{E}\!\left[
X_{t}(0)\big|\,\mathcal{F}_{\infty}\right]  \quad\text{in }\mathcal{L}_{1\,},
\end{equation}
where $\mathcal{F}_{\infty}:=\sigma\left(  \cup_{n>1/t}\mathcal{F}_{\tau_{n}%
}\right)  $.

Noting that $\tau_{n}\uparrow t$, we conclude that
\begin{equation}
\mathcal{F}_{t-}\subseteq\mathcal{F}_{\infty}\subseteq\mathcal{F}_{t\,}.
\end{equation}
Since $X_{\cdot}(0)$ is continuous at fixed $t$ a.s., we have $X_{t}%
(0)=\mathbf{E}\!\left[  X_{t}(0)\,\big|\,\mathcal{F}_{t-}\right]  $ almost
surely. Consequently, $\mathbf{E}\!\left[  X_{t}(0)\,\big|\,\mathcal{F}%
_{\infty}\right]  =X_{t}(0)$ almost surely, and we get, as a result,
\begin{equation}
\mathbf{E}\!\left[  X_{t}(0)\,\big|\,\mathcal{F}_{\tau_{n}}\right]
\,\underset{n\uparrow\infty}{\longrightarrow}\,X_{t}(0)\quad\text{in
}\mathcal{L}_{1\,}. \label{n3}%
\end{equation}
Combining (\ref{n2}) and (\ref{n3}), we have
\begin{equation}
\mathbf{E}\!\left[  X_{t}(x_{n})\,\big|\,\mathcal{F}_{\tau_{n}}\right]
\,\underset{n\uparrow\infty}{\longrightarrow}\,X_{t}(0)\quad\text{in
}\mathcal{L}_{1\,}.
\end{equation}
\textrm{F}rom this convergence and (\ref{n1}) we get finally%
\begin{equation}
\mathbf{E}\Big(\mathsf{1}_{A}\left\vert X_{t}(0)-S_{t-\tau_{n}}^{\alpha
}X_{\tau_{n}}(x_{n})\right\vert \Big)\,\underset{n\uparrow\infty
}{\longrightarrow}\,0.
\end{equation}
Since $\,S_{t-\tau_{n}}^{\alpha}X_{\tau_{n}}(x_{n})\leq\theta+1/n$%
\thinspace\ on $A$, for all $n>1/t$, the latter convergence implies that
$X_{t}(0)\leq\theta$ almost surely on the event $A$. Thus, the proof is finished.
\end{proof}

\begin{lemma}
[\textbf{Some local boundedness}]\label{n.L2} Fix any non-empty bounded
$B\subset\mathsf{R}$. Then
\begin{equation}
W_{B}:=\ \sup_{(c,s,x):\ \,c\geq1,\,\,0\vee(t-c^{-\alpha})\leq s<t,\,\,x\in
B}\frac{X_{s}\bigl(B_{c\,(t-s)^{1/\alpha}}(x)\bigr)}{c\,(t-s)^{1/\alpha}%
}<\infty\quad\text{a.s.}%
\end{equation}
{}
\end{lemma}

\begin{proof}
Every ball of radius $c\,(t-s)^{1/\alpha}$ can be covered with at most $[c]+1$
balls of radius $(t-s)^{1/\alpha}$. Therefore,
\begin{align}
&  \sup_{(c,s,x):\ \,c\geq1,\,\,0\vee(t-c^{-1/\alpha})\leq s<t,\,\,x\in
B}\frac{X_{s}\bigl(B_{c\,(t-s)^{1/\alpha}}(x)\bigr)}{c\,(t-s)^{1/\alpha}%
}\,\nonumber\\
&  \leq\ 2\,\sup_{(s,x):\ \,0<s\leq t,\,x\in B_{1}}\frac{X_{s}%
\bigl(B_{(t-s)^{1/\alpha}}(x)\bigr)}{(t-s)^{1/\alpha}}\,,
\end{align}
where $B_{1}:=\bigl\{x:\ \mathrm{dist}(x,\overline{B})\leq1\bigr\}$ with
$\overline{B}$ denoting the closure of $B$. (The restriction $s\geq
t-c^{-1/\alpha}$ is imposed to have all centers $x$ of the balls
$B_{(t-s)^{1/\alpha}}(x)$ in $B_{1\,}.)$ We further note that
\begin{equation}
S_{t-s}^{\alpha}X_{s}\,(x)=\int_{\mathsf{R}}\mathrm{d}y\ p_{t-s}^{\alpha
}(x-y)X_{s}(y)\,\geq\int_{B_{(t-s)^{1/\alpha}}(x)}\mathrm{d}y\ p_{t-s}%
^{\alpha}(x-y)X_{s}(y).
\end{equation}
Using the monotonicity and the scaling property of $p^{\alpha}$, we get the
bound
\begin{equation}
S_{t-s}^{\alpha}X_{s}\,(x)\,\geq(t-s)^{-1/\alpha}p_{1}^{\alpha}(1)X_{s}%
\bigl(B_{(t-s)^{1/\alpha}}(x)\bigr).
\end{equation}
Consequently,
\begin{equation}
\sup_{(s,x):\ \,0<s\leq t,\,x\in B_{1}}\frac{X_{s}\bigl(B_{(t-s)^{1/\alpha}%
}(x)\bigr)}{(t-s)^{1/\alpha}}\,\leq\,\frac{1}{p_{1}^{\alpha}(1)}%
\ \sup_{(s,x):\ \,0<s\leq t,\,x\in B_{1}}S_{t-s}^{\alpha}X_{s}\,(x).
\end{equation}
It was proved in Lemma~2.11 of \cite{FleischmannMytnikWachtel2010.optimal.Ann}%
, that the random variable at the right hand side is finite. Thus, the lemma
is proved.
\end{proof}

Introduce the event
\begin{equation}
D_{\theta}:=\left\{  X_{t}(0)>\theta,\ \sup_{0<s\leq t}X_{s}(\mathsf{R}%
)\leq\theta^{-1},\ W_{B_{3}(0)}\leq\theta^{-1}\right\}  \!.
\end{equation}
For the rest of the paper take an arbitrary $\varepsilon\in(0,t\wedge1/8)$.
For constants $c,Q>0,$ define the stopping time
\begin{align}
\tau_{\varepsilon,c,Q}\,:=\,\inf\bigg\{s\in(t-\varepsilon,t):\ \Delta
X_{s}(y)  &  >Q\,\bigl(y\,(t-s)\bigr)^{1/(1+\beta)}\log^{1/(1+\beta
)}\bigl((t-s)^{-1}\bigr),\nonumber\\
&  \,\,\text{for some \ }\frac{c}{2}\,(t-s)^{1/\alpha}\leq y\leq\frac{3c}%
{2}\,(t-s)^{1/\alpha}\bigg\}.
\end{align}
In the next lemma we are going to show the finiteness of \thinspace
$\tau_{\varepsilon,c,Q\,},$\thinspace\ which means that there is a
\textquotedblleft big\textquotedblright\ jump close to time $t$ and to the
spatial point $z=0$.

\begin{lemma}
[\textbf{Finiteness of} $\,\tau_{\varepsilon,c_{(\ref{n.L3.f})},Q}$%
]\label{n.L3} For each $\theta>0$ there exists a constant $c_{(\ref{n.L3.f}%
)}=c_{(\ref{n.L3.f})}(\theta)\geq1$ such that
\begin{equation}
\mathbf{P}\!\left(  \tau_{\varepsilon,c_{(\ref{n.L3.f})},Q}=\infty
\,\big|\,D_{\theta}\right)  =\,0,\quad\varepsilon\in(0,t\wedge1/8),\,\ Q>0.
\label{n.L3.f}%
\end{equation}
{}
\end{lemma}

\begin{proof}
Analogously to the proof of Lemma~4.3 in
\cite{FleischmannMytnikWachtel2010.optimal.Ann}, to demonstrate that the
number of jumps is greater than zero almost surely on some event, it is enough
to show divergence of a certain integral on that event or even on a bigger
one. Specifically here, it suffices to verify that%
\begin{equation}
I_{\varepsilon,c}\,:=\int_{t-\varepsilon}^{t}\frac{\mathrm{d}s}{(t-s)\log
\bigl((t-s)^{-1}\bigr)}\int_{\frac{c}{2}\,(t-s)^{1/\alpha}}^{\frac{3c}%
{2}\,(t-s)^{1/\alpha}}\mathrm{d}y\ y^{-1}X_{s}(y)=\infty
\end{equation}
almost surely on the event $\left\{  X_{t}(0)>\theta,\ \sup_{0<s\leq t}%
X_{s}(\mathsf{R})\leq\theta^{-1}\right\}  $.

The mapping $\varepsilon\mapsto I_{\varepsilon,c}$ is non-increasing.
Therefore we shall additionally assume, without loss of generality, that
$\varepsilon\leq c^{-1/\alpha}$ and this in turn implies that
$c\,(t-s)^{1/\alpha}\leq1$ for all $s\in(t-\varepsilon,t)$. So, in what
follows, in the proof of the lemma we will assume without loss of generality
that given $c$, we choose $\varepsilon$ so that,
\begin{equation}
c\,(t-s)^{1/\alpha}\leq1,\quad s\in(t-\varepsilon,t).
\end{equation}

Since $y\leq\frac{3c}{2}\,(t-s)^{1/\alpha}$ and $p_{s}^{\alpha}(x)\leq
p_{s}^{\alpha}(0)$ for all $x\in\mathsf{R}$, we have
\begin{align}
I_{\varepsilon,c}\,\geq\frac{2}{3c}\int_{t-\varepsilon}^{t} &  \,\frac
{\mathrm{d}s}{(t-s)^{1+1/\alpha}\log\bigl((t-s)^{-1}\bigr)}\nonumber\\
&  \times\int_{\frac{c}{2}\,(t-s)^{1/\alpha}}^{\frac{3c}{2}\,(t-s)^{1/\alpha}%
}\!\!\mathrm{d}y\ \frac{p_{t-s}^{\alpha}\bigl(c\,(t-s)^{1/\alpha}%
-y\bigr)}{p_{t-s}^{\alpha}(0)}\,X_{s}(y).
\end{align}
Then, using the scaling property of $p^{\alpha}$, we obtain
\begin{align}
I_{\varepsilon,c}\,\geq\frac{2}{3c\,p_{1}^{\alpha}(0)}\int_{t-\varepsilon}^{t}
&  \,\frac{\mathrm{d}s}{(t-s)\log\bigl((t-s)^{-1}\bigr)}\,\biggl(S_{t-s}%
^{\alpha}X_{s}\,\bigl(c\,(t-s)^{1/\alpha}\bigr)\nonumber\label{n4}\\
- &  \int_{\left\vert _{\!_{\!_{\,}}}y-c\,(t-s)^{1/\alpha}\right\vert
\,>\,\frac{c}{2}\,(t-s)^{1/\alpha}}\mathrm{d}y\ p_{t-s}^{\alpha}%
\bigl(c\,(t-s)^{1/\alpha}-y\bigr)X_{s}(y)\!\biggr).
\end{align}
Since we are in dimension one, if
\begin{align}
y &  \in\widetilde{D}_{s,j}:=\nonumber\\
&  \left\{  z:\;c\,\Big(\frac{1}{2}+j\Big)(t-s)^{1/\alpha}<\left\vert
_{\!_{\!_{\,}}}z-c\,(t-s)^{1/\alpha}\right\vert <c\,\Big(2+\frac{1}%
{2}+j\Big)(t-s)^{1/\alpha}\right\}  \!,\label{1dim}%
\end{align}
then%
\begin{align}
&  p_{t-s}^{\alpha}\bigl(c\,(t-s)^{1/\alpha}-y\bigr)\leq p_{t-s}^{\alpha
}\bigl(c\,(j+1/2)(t-s)^{1/\alpha}\bigr)\\
&  =(t-s)^{-1/\alpha}p_{1}^{\alpha}\bigl(c\,(j+1/2)\bigr)\leq c_{(\ref{tail}%
)}c^{-\alpha-1}(t-s)^{-1/\alpha}(1/2+j)^{-\alpha-1}.\nonumber
\end{align}
\textrm{F}rom this bound
we conclude that
\begin{gather}
\int_{\left\vert _{\!_{\!_{\,}}}y-c\,(t-s)^{1/\alpha}\right\vert \,>\,\frac
{c}{2}\,(t-s)^{1/\alpha}}\mathrm{d}y\ p_{t-s}^{\alpha}\bigl(c\,(t-s)^{1/\alpha
}-y\bigr)\mathsf{1}_{B_{2}(0)}(y)X_{s}(y)\nonumber\\
\leq c_{(\ref{tail})}c^{-\alpha-1}(t-s)^{-1/\alpha}\sum_{j=0}^{\infty
}(1/2+j)^{-\alpha-1}\int_{\widetilde{D}_{s,j}}\mathrm{d}y\ \mathsf{1}%
_{B_{2}(0)}(y)X_{s}(y).
\end{gather}
Now recall again that the spatial dimension equals to one and hence for any
$j\geq0$ the set $\widetilde{D}_{s,j}$ in (\ref{1dim}) is the union of two
balls of radius $c\,(t-s)^{1/\alpha}$. If furthermore $\widetilde{D}_{s,j}\cap
B_{2}(0)\neq\emptyset$, then, in view of the assumption $c\,(t-s)^{1/\alpha
}\leq1$, the centers of those balls lie in $B_{3}(0)$. Therefore we can apply
Lemma~\ref{n.L2} to bound the integral $\int_{\widetilde{D}_{s,j}}%
\mathrm{d}y\ \mathsf{1}_{B_{2}(0)}(y)X_{s}(y)$ by $2c\,(t-s)^{1/\alpha
}W_{B_{3}(0)}$
and obtain%
\begin{gather}
\int_{\left\vert _{\!_{\!_{\,}}}y-c\,(t-s)^{1/\alpha}\right\vert \,>\,\frac
{c}{2}\,(t-s)^{1/\alpha}}\mathrm{d}y\ p_{t-s}^{\alpha}\bigl(c\,(t-s)^{1/\alpha
}-y\bigr)\mathsf{1}_{B_{2}(0)}(y)X_{s}(y)\nonumber\\
\leq2W_{B_{3}(0)}c_{(\ref{tail})}c^{-\alpha}\sum_{j=0}^{\infty}%
(1/2+j)^{-\alpha-1}\leq CW_{B_{3}(0)}c^{-\alpha}.\label{n5}%
\end{gather}
Furthermore, if $|y|\geq2$ and $(t-s)\leq c^{-\alpha}$, then
\begin{equation}
p_{t-s}^{\alpha}\bigl(c\,(t-s)^{1/\alpha}-y\bigr)\leq p_{t-s}^{\alpha
}(1)=(t-s)^{-1/\alpha}p_{1}^{\alpha}\bigl((t-s)^{-1/\alpha}\bigr)\leq
c_{(\ref{tail})}(t-s).
\end{equation}
This implies that
\begin{gather}
\int_{\mathsf{R}\setminus B_{2}(0)}\mathrm{d}y\ p_{t-s}^{\alpha}%
\bigl(c\,(t-s)^{1/\alpha}-y\bigr)X_{s}(y)\leq c_{(\ref{tail})}(t-s)X_{s}%
(\mathsf{R})\nonumber\\
\leq c_{(\ref{tail})}c^{-\alpha}X_{s}(\mathsf{R}).
\end{gather}
Combining this bound with (\ref{n5}), we obtain
\begin{gather}
\int_{\left\vert _{\!_{\!_{\,}}}y-c\,(t-s)^{1/\alpha}\right\vert \,>\,\frac
{c}{2}\,(t-s)^{1/\alpha}}\mathrm{d}y\ p_{t-s}^{\alpha}\bigl(c\,(t-s)^{1/\alpha
}-y\bigr)X_{s}(y)\\
\leq\,Cc^{-\alpha}\Bigl(W_{B_{3}(0)}+\sup_{0<s\leq t}X_{s}(\mathsf{R}%
)\Bigr).\nonumber
\end{gather}
Thus, we can choose $c$ so large that the right hand side in the previous
inequality does not exceed $\theta/2$. Since, in view of Lemma~\ref{n.L1},
\begin{equation}
\liminf_{s\uparrow t}S_{t-s}^{\alpha}X_{s}\bigl(c\,(t-s)^{1/\alpha
}\bigr)>\theta,
\end{equation}
we finally get%
\begin{align}
&  \liminf_{s\uparrow t}\Biggl(S_{t-s}^{\alpha}X_{s}\bigl(c\,(t-s)^{1/\alpha
}\bigr)\\
&  \qquad\quad-\int_{|y-c\,(t-s)^{1/\alpha}|\,>\,\frac{c}{2}\,(t-s)^{1/\alpha
}}\mathrm{d}y\ p_{t-s}^{\alpha}\bigl(c\,(t-s)^{1/\alpha}-y\bigr)X_{s}%
(y)\Biggr)\geq\theta/2.\nonumber
\end{align}
\textrm{F}rom this bound and (\ref{n4}) the desired property of
$I_{\varepsilon,c}$ follows.
\end{proof}

Fix any $\theta>0$, and to simplify notation write $c:=c_{(\ref{n.L3.f})}.$
For all $n$ sufficiently large, say $n\geq N_{0\,},$ define
\begin{align}
A_{n}\,:=\biggl\{\Delta X_{s}\left(  \Big(\frac{c}{2}\,2^{-n},\frac{3c}%
{2}\,2^{-n}\Big)\right)   &  \geq2^{-(\bar{\eta}_{\mathrm{c}}+1)n}%
\,n^{1/(1+\beta)}\\
\text{ }  &  \text{for some }s\in(t-2^{-\alpha n},t-2^{-\alpha(n+1)}%
)\biggr\}.\nonumber
\end{align}

Based on Lemma~\ref{n.L3} we will show in the following lemma that,
conditionally on $D_{\theta\,},$\thinspace\ infinitely many of the $A_{n}$'s
occur. This then gives us a bit more precise information on the
\textquotedblleft big\textquotedblright\ jumps we are looking for.

\begin{lemma}
[\textbf{Existence of big jumps}]\label{n.L4}We have
\begin{equation}
\mathbf{P}\!\left(  A_{n}\,\text{infinitely often}\ \big|\,D_{\theta}\right)
=1.
\end{equation}
{}
\end{lemma}

\begin{proof}
If $y\in\left(  \frac{c}{2}\,(t-s)^{1/\alpha},\,\frac{3c}{2}\,(t-s)^{1/\alpha
}\right)  $ and $s\in(t-2^{-\alpha n},\,t-2^{-\alpha(n+1)})$, then%
\begin{gather}
\left(  (t-s)y\log\bigl((t-s)^{-1}\bigr)\right)  ^{1/(1+\beta)}\geq\left(
2^{-\alpha(n+1)}\,\frac{c}{2}\,2^{-n-1}\alpha n\log2\right)  ^{1/(1+\beta
)}\nonumber\\
=c_{(\ref{100})}^{-1}2^{-(\bar{\eta}_{\mathrm{c}}+1)n}\,n^{1/(1+\beta)}.
\label{100}%
\end{gather}
This implies that
\begin{align}
A_{n}\supseteq\Bigg\{\Delta X_{s}  &  \biggl(\!\Big(\frac{c}{2}%
\,(t-s)^{1/\alpha},\,\frac{3c}{2}\,(t-s)^{1/\alpha}\Big)\!\biggr)\label{101}\\
&  \geq c_{(\ref{100})}\Big((t-s)y\log\bigl((t-s)^{-1}\bigr)\Big)^{1/(1+\beta
)}\nonumber\\
&  \qquad\qquad\text{for some }s\in(t-2^{-\alpha n},\,t-2^{-\alpha
(n+1)})\Bigg\}.\nonumber
\end{align}
In what follows we denote with some abuse of notation $\,\tau_{\varepsilon
,c}:=\tau_{\varepsilon,c,c_{(\ref{100})}}.$\thinspace\ Consequently, from
(\ref{101}) we get
\begin{equation}
\bigcup_{n=N}^{\infty}A_{n}\supseteq\{\tau_{2^{-\alpha N},c}<\infty
\}\quad\text{for all }N>N_{0}\vee\alpha^{-1}\log_{2}(t\wedge1/8).
\end{equation}
Applying Lemma~\ref{n.L3} and using the monotonicity of the union in $N$, we
get
\begin{equation}
\mathbf{P}\!\left(  \bigcup_{n=N}^{\infty}A_{n}\big|\,D_{\theta}\right)
=1\quad\text{for all }N\geq N_{0\,}.
\end{equation}
This completes the proof.
\end{proof}

Now it is time to explain our

\begin{proof}
[Detailed strategy of proof of Theorem~\emph{\ref{T.prop.dens.fixed}(b)}%
]Define
\begin{align}
&  A^{\varepsilon}:=\label{Aep}\\
&  \Big\{\Delta X_{s}(y)\leq c_{(\ref{in6.1a})}%
\,\bigl((t-s)|y|\bigr)^{1/(1+\beta)}(f_{s,x})^{\ell}\text{ for all
}s<t\ \text{and}\ y\in B_{1/\mathrm{e}}(0)\Big\}\nonumber\\
&  \cap\,\Big\{\Delta X_{s}(y)\leq c^{\ast}(t-s)^{1/(1+\beta)-\gamma
}\ \text{for all }s<t\ \text{and }y\in\mathsf{R}\Big\}\cap\{V\leq
c_{\varepsilon}\},\nonumber
\end{align}
where $f_{s,x\,},$ $\ell$ and $c^{\ast}$ are defined in~(\ref{def_f}),
(\ref{def.lambda1}) and (\ref{6.A.eps2}), respectively. Note that $D_{\theta
}\uparrow\{X_{t}(0)>0\}$ as $\theta\downarrow0$ and by (\ref{def_ve}),
(\ref{6.A.eps2}) and Lemma~\ref{6.00} we have $A^{\varepsilon}\uparrow\Omega$
as $\varepsilon\downarrow0$. Hence, for the proof of
Theorem~\ref{T.prop.dens.fixed}(b) \emph{it is sufficient}\/ to show that%
\begin{equation}
\mathbf{P}\biggl(\,\sup_{x\in B_{\epsilon}(0),\,x\neq0}\frac{\bigl|\tilde
{X}_{t}(x)-\tilde{X}_{t}(0)\bigr|}{|x|^{\bar{\eta}_{\mathrm{c}}}}%
\,=\,\infty\,\bigg|\,D_{\theta}\cap A^{\varepsilon}\biggr)=1.
\end{equation}
Moreover, since $Z_{t}^{1}$ and $Z_{t}^{3}$ are a.s. Lipschitz continuous at
$0$, the latter will follow from the equality%
\begin{align}
\mathbf{P}\Big(Z_{t}^{2}(c\,2^{-n-2})-Z_{t}^{2}(0)\geq &  \ 2^{-\bar{\eta
}_{\mathrm{c}}n}\,n^{1/(1+\beta)-\varepsilon}\text{ }\label{Nr.}\\
&  \ \text{infinitely often}\,\Big|\,D_{\theta}\cap A^{\varepsilon
}\Big)=1.\nonumber
\end{align}
To verify (\ref{Nr.}), we will again exploit our method of representing
$Z_{t}^{2}$ using a time-changed stable process. To be more precise, applying
(\ref{6.T1.1}) with $x=c\,2^{-n-2}$ (for $n$ sufficiently large) and using
$n$-dependent notation as $L_{n}^{\pm},T_{n,\pm}$ (and later $\varphi_{n,\pm
})$, we have
\begin{equation}
Z_{t}^{2}(c\,2^{-n-2})-Z_{t}^{2}(0)=L_{n}^{+}(T_{n,+})-L_{n}^{-}(T_{n,-}).
\end{equation}
Let us define the following events
\[
B_{n}^{+}:=\bigl\{L_{n}^{+}(T_{n,+})\geq2^{1-\bar{\eta}_{\mathrm{c}}%
n}\,n^{1/(1+\beta)-\varepsilon}\bigr\},\quad B_{n}^{-}:=\bigl\{L_{n}%
^{-}(T_{n,-})\leq2^{-\bar{\eta}_{\mathrm{c}}n}\,n^{1/(1+\beta)-\varepsilon
}\bigr\}
\]
and
\begin{equation}
B_{n}:=B_{n}^{+}\cap B_{n}^{-}.
\end{equation}
Then, obviously,
\begin{equation}
\left\{  Z_{t}^{2}(c\,2^{-n-2})-Z_{t}^{2}(0)\geq2^{-\bar{\eta}_{\mathrm{c}}%
n}\,n^{1/(1+\beta)-\varepsilon}\right\}  \supseteq B_{n}\supseteq B_{n}\cap
A_{n\,}.
\end{equation}
Thus, (\ref{Nr.}) will follow once we verify
\begin{equation}
\lim_{N\uparrow\infty}\mathbf{P}\biggl(\,\bigcup_{n=N}^{\infty}(B_{n}\cap
A_{n})\,\bigg|\,D_{\theta}\cap A^{\varepsilon}\,\biggr)=\,1. \label{n6}%
\end{equation}
Taking into account Lemma~\ref{n.L4}, we conclude that to get (\ref{n6}) we
have to show
\begin{equation}
\lim_{N\uparrow\infty}\mathbf{P}\biggl(\,\bigcup_{n=N}^{\infty}(B_{n}%
^{\mathrm{c}}\cap A_{n})\,\bigg|\,D_{\theta}\cap A^{\varepsilon}\biggr)=\,0.
\label{n7}%
\end{equation}
Hence, the proof of Theorem~\ref{T.prop.dens.fixed}(b) will be complete once
we demonstrated statement (\ref{n7}).
\end{proof}

Now we will present two lemmas, from which (\ref{n7}) will follow immediately.
To this end, split
\begin{equation}
B_{n}^{\mathrm{c}}\cap A_{n}=(B_{n}^{+,\mathrm{c}}\cap A_{n})\cup
(B_{n}^{-,\mathrm{c}}\cap A_{n}). \label{114}%
\end{equation}

\begin{lemma}
[\textbf{First term in }(\ref{114})]\label{n.L5}We have%
\begin{equation}
\lim_{N\uparrow\infty}\sum_{n=N}^{\infty}\mathbf{P}\bigl(B_{n}^{+,\mathrm{c}%
}\cap A_{n}\cap A^{\varepsilon}\bigr)=\,0.
\end{equation}
{}
\end{lemma}

The proof of this lemma is a word-for-word repetition of the proof of
Lemma~5.3 in \cite{FleischmannMytnikWachtel2010.optimal.Ann} (it is even
simpler as we do not need additional indexing in $k$ here), and we omit it.
The idea behind the proof is simple: Whenever $X$ has a \textquotedblleft
big\textquotedblright\ jump guaranteed by $A_{n\,},$ this jump corresponds to
the jump of $L_{n}^{+}$ and then it is very difficult for a spectrally
positive process $L_{n}^{+}$ to come down, which is required by $B_{n}%
^{+,\mathrm{c}}$.

\begin{lemma}
[\textbf{Second term in }(\ref{114})]\label{n.L5'}We have%
\begin{equation}
\lim_{N\uparrow\infty}\sum_{n=N}^{\infty}\mathbf{P}\bigl(B_{n}^{-,\mathrm{c}%
}\cap A_{n}\cap A^{\varepsilon}\cap D_{\theta}\bigr)=0.
\end{equation}
{}
\end{lemma}

The remaining part of the paper will be devoted to the proof of
Lemma~\ref{n.L5'} and we prepare now for it.

One can easily see that $B_{n}^{-,\mathrm{c}}$ is a subset of a union of two
events (with the obvious correspondence):
\begin{align}
B_{n}^{-,\mathrm{c}}\subseteq U_{n}^{1}\cup U_{n}^{2}:=\,  &  \bigl\{\Delta
L_{n}^{-}>2^{-\bar{\eta}_{\mathrm{c}}n}\,n^{1/(1+\beta)-2\varepsilon
}\bigr\}\nonumber\\
&  \cup\bigl\{\Delta L_{n}^{-}\leq2^{-\bar{\eta}_{\mathrm{c}}n}\,n^{1/(1+\beta
)-2\varepsilon},\ L_{n}^{-}(T_{n,-})>2^{-\bar{\eta}_{\mathrm{c}}%
n}\,n^{1/(1+\beta)-\varepsilon}\bigr\},
\end{align}
where
\begin{equation}
\Delta L_{n}^{-}:=\sup_{0<s\leq T_{n,-}}\Delta L_{n}^{-}(s).
\end{equation}
The occurrence of the event $U_{n}^{1}$ means that $L_{n}^{-}$ has big jumps.
If $U_{n}^{2}$ occurs, it means that $L_{n}^{-}$ gets large without big jumps.
It is well-known that stable processes without big jumps can not achieve large
values. Thus, the statement of the next lemma is not surprising.

\begin{lemma}
[\textbf{No big values of} $L_{n}^{-}$ \textbf{in case of absence of
\textquotedblleft big\textquotedblright\ jumps}]\label{n.L6}\hfill We\newline
have%
\begin{equation}
\lim_{N\uparrow\infty}\sum_{n=N}^{\infty}\mathbf{P}(U_{n}^{2}\cap
A^{\varepsilon})=0.
\end{equation}
{}
\end{lemma}

We omit the proof of this lemma as well, since its crucial part related to
bounding of $\mathbf{P}(U_{n}^{2}\cap A^{\varepsilon})$ is a repetition of the
proof of Lemma~5.6 in \cite{FleischmannMytnikWachtel2010.optimal.Ann} (again
with obvious simplifications).

\begin{lemma}
[\textbf{Big jumps of} $L_{n}^{-}$ \textbf{caused by several big jumps of}
$M$]\label{n.L7}\hfill There \newline exist constants $\rho$ and $\xi$ such
that, for all sufficiently large values of \thinspace$n$,
\begin{equation}
A^{\varepsilon}\cap A_{n}\cap U_{n}^{1}\,\subseteq\,A^{\varepsilon}\cap
E_{n}(\rho,\xi),
\end{equation}
where
\begin{align}
E_{n}(\rho,\xi):=  &  \left\{  _{\!_{\!_{\,_{_{\!_{\!_{_{\!_{\!_{_{\!_{\!_{\,}%
}}\,}}}}}}}}}}\mathrm{There}\text{ }\mathrm{exist}\text{ }\mathrm{at}\text{
}\mathrm{least}\text{\textrm{ }}\mathrm{two}\text{\textrm{ }}\mathrm{jumps}%
\text{\textrm{ }}\mathrm{of\;}M\ \mathrm{of}\text{\textrm{ }}\mathrm{the}%
\text{\textrm{ }}\mathrm{form\;}r\delta_{(s,y)}\;\mathrm{such}\text{\textrm{
}}\mathrm{that}\right. \nonumber\\
&  \quad r\geq\left(  (t-s)\max\bigl\{(t-s)^{1/\alpha},|y|\bigr\}\right)
^{1/(1+\beta)}\log^{1/(1+\beta)-2\varepsilon}\bigl((t-s)^{-1}%
\bigr),\nonumber\\[3pt]
&  \quad|y|\leq(t-s)^{1/\alpha}\log^{\xi}\bigl((t-s)^{-1}\bigr),\ \!\!\left.
_{\!_{\!_{\,_{_{\!_{\!_{_{\!_{\!_{_{\!_{\!_{\,}}}\,}}}}}}}}}}s\in\left[
t-2^{-\alpha n}\,n^{\rho},\ t-2^{-\alpha n}\,n^{-\rho}\right]  \right\}  \!.
\end{align}
{}
\end{lemma}

\begin{proof}
By the definition of $A_{n\,},$\thinspace\ there exists a jump of $M$ of the
form $r\delta_{(s,y)}$ with $r,s$ as in $E_{n}(\rho,\xi)$, and $y>c\,2^{-n-1}%
$. Furthermore, noting that $\varphi_{n,-}(y)=0$ for $y\geq c\,2^{-n-3},$ we
see that the jumps $r\delta_{(s,y)}$ of $M$ contribute to $L_{n}^{-}(T_{n,-})$
if and only if $y<c\,2^{-n-3}$. Thus, to prove the lemma it is sufficient to
show that $U_{n}^{1}$ yields the existence of at least one further jump of $M$
on the half-line $\{y<c\,2^{-n-3}\}$ with properties mentioned in the
statement. Denote
\begin{align}
D:=\bigg\{(r,s,y)\!:\  &  r\geq\!\left(  (t-s)\max\bigl\{(t-s)^{1/\alpha
},|y|\bigr\}\right)  ^{1/(1+\beta)}\log^{1/(1+\beta)-2\varepsilon
}\bigl((t-s)^{-1}\bigr),\nonumber\\
&  y\in\left(  -(t-s)^{1/\alpha}\log^{\xi}\bigl((t-s)^{-1}\bigr),\,c\,2^{-n-3}%
\right)  \!,\ \nonumber\\
&  \qquad\qquad\qquad\qquad\qquad\qquad\quad\ \,s\in\left[  t-2^{-\alpha
n}\,n^{\rho},\ t-2^{-\alpha n}\,n^{-\rho}\right]  \!\bigg\}. \label{setD}%
\end{align}
Then we need to show that $U_{n}^{1}$ implies the existence of a jump
$r\delta_{(s,y)}$ of $M$ with $(r,s,y)\in D.$

Note that
\begin{align*}
\lefteqn{D=D_{1}\cap D_{2}\cap D_{3}}\\
&  :=\left\{  (r,s,y):\;r\geq0,\ s\in\lbrack0,t],\;y\in\left(
-(t-s)^{1/\alpha}\log^{\xi}\bigl((t-s)^{-1}\bigr),\,c\,2^{-n-3}\right)
\right\} \\
&  \quad\quad\cap\Big\{(r,s,y):\;r\geq0,\ y\in(-\infty,\,c\,2^{-n-3}%
),\ s\in\left[  t-2^{-\alpha n}\,n^{\rho},\ t-2^{-\alpha n}\,n^{-\rho}\right]
\!\Big\}\\
&  \quad\quad\cap\bigg\{(r,s,y):\;y\in(-\infty,\,c\,2^{-n-3}),\ s\in
\lbrack0,t],\\
&  \;\;\;\;\;\;\;\quad\qquad\ r\geq\left(  (t-s)\max\bigl\{(t-s)^{1/\alpha
},|y|\bigr\}\right)  ^{1/(1+\beta)}\log^{1/(1+\beta)-2\varepsilon
}\bigl((t-s)^{-1}\bigr)\bigg\}.
\end{align*}
Therefore,%
\begin{equation}
D^{\mathrm{c}}\cap\{y<c2^{-n-3}\}\,=\,\left(  D_{1}^{\mathrm{c}}%
\cap\{y<c\,2^{-n-3}\}\right)  \cup\left(  D_{1}\cap D_{2}^{\mathrm{c}}\right)
\cup\left(  D_{1}\cap D_{2}\cap D_{3}^{\mathrm{c}}\right)  \!,
\end{equation}
where the complements are defined with respect to the set%
\begin{equation}
\left\{  _{\!_{\!_{\,}}}(r,s,y):\;r\geq0,\ s\in\lbrack0,t],\;y\in
\mathsf{R}\right\}  \!.
\end{equation}

We first show that any jumps of $M$ in $D_{1}^{\mathrm{c}}\cap\{y<c\,2^{-n-3}%
\}$ cannot be the course of a jump of $L_{n}^{-}$ such that $U_{n}^{1}$ holds.
Indeed, using Lemma~\ref{L1old} with $\delta=\bar{\eta}_{\mathrm{c}\,}%
,$\thinspace\ we get for $y<0$ the inequality
\begin{gather}
\varphi_{n,-}(y)=p_{t-s}^{\alpha}(y)-p_{t-s}^{\alpha}(y-c\,2^{-n-1}%
)\leq2^{1-\bar{\eta}_{\mathrm{c}}n}(t-s)^{-\bar{\eta}_{\mathrm{c}}/\alpha
}p_{t-s}^{\alpha}(y)\nonumber\\
\leq C\,2^{-\bar{\eta}_{\mathrm{c}}n}(t-s)^{-(1+\bar{\eta}_{\mathrm{c}%
})/\alpha}\Big(\frac{y}{(t-s)^{1/\alpha}}\Big)^{\!-\alpha-1}\nonumber\\
=\,C\,2^{-\bar{\eta}_{\mathrm{c}}n}(t-s)^{1-\bar{\eta}_{\mathrm{c}}/\alpha
}|y|^{-\alpha-1}, \label{n8}%
\end{gather}
in the second step we used the scaling property and (\ref{tail}).

Further, by~(\ref{Aep}), on the set $A^{\varepsilon}$ we have
\begin{equation}
\Delta X_{s}(y)\leq C\,\bigl(|y|(t-s)\bigr)^{1/(1+\beta)}(f_{s,y})^{\ell
},\quad|y|\leq1/\mathrm{e}, \label{n9}%
\end{equation}
and
\begin{equation}
\Delta X_{s}(y)\leq C\,(t-s)^{1/(1+\beta)-\gamma},\quad|y|>1/\mathrm{e},
\label{n10}%
\end{equation}
and recall that $f_{s,x}=\,\log\bigl((t-s)^{-1}\bigr)\,\mathsf{1}_{\{x\neq
0\}}\log\bigl(|x|^{-1}\bigr)$. Combining (\ref{n8}) and~(\ref{n9}), we
conclude that the corresponding jump of $L_{n}^{-},$ henceforth denoted by
\hspace{-1pt}$\Delta L_{n}^{-}[r\delta_{(s,y)}],$ is bounded by
\begin{equation}
C\,2^{-\bar{\eta}_{\mathrm{c}}n}(t-s)^{1-\bar{\eta}_{\mathrm{c}}/\alpha
+\frac{1}{1+\beta}}\log^{\frac{1}{1+\beta}+q}\bigl((t-s)^{-1}%
\bigr)|y|^{-\alpha-1+\frac{1}{1+\beta}}\log^{\frac{1}{1+\beta}+q}%
\bigl(|y|^{-1}\bigr).
\end{equation}
Since $|y|^{-\alpha-1+\frac{1}{1+\beta}}\log^{\frac{1+\gamma}{1+\beta}%
}\bigl(|y|^{-1}\bigr)$ is monotone decreasing, we get, maximizing over $y$,
for $y<-(t-s)^{1/\alpha}\log^{\xi}\bigl((t-s)^{-1}\bigr)$ the bound
\begin{equation}
\Delta L_{n}^{-}[r\delta_{(s,y)}]\leq C\,2^{-\bar{\eta}_{\mathrm{c}}n}%
\log^{\frac{2}{1+\beta}+2q-\xi(\alpha+1-\frac{1}{1+\beta})}\bigl(|y|^{-1}%
\bigr).
\end{equation}
Choosing $\xi\geq\frac{2+2q(1+\beta)}{(1+\beta)(1+\alpha)-1}$\thinspace, we
see that
\begin{equation}
\Delta L_{n}^{-}[r\delta_{(s,y)}]\leq C\,2^{-\bar{\eta}_{\mathrm{c}}n}%
,\quad|y|<1/\mathrm{e}. \label{nn1}%
\end{equation}
Moreover, if $y<-1/\mathrm{e}$, then it follows from (\ref{n8}) and
(\ref{n10}) that the jump $\Delta L_{n}^{-}[r\delta_{(s,y)}]$ is bounded by
\begin{equation}
C\,2^{-\bar{\eta}_{\mathrm{c}}n}(t-s)^{1-\bar{\eta}_{\mathrm{c}}/\alpha
+\frac{1}{1+\beta}-\gamma}|y|^{-\alpha-1}\leq C\,2^{-\bar{\eta}_{\mathrm{c}}%
n}. \label{nn2}%
\end{equation}
Combining (\ref{nn1}) and (\ref{nn2}), we see that all the jumps of $M$ in
$D_{1}^{\mathrm{c}}\cap\{y<c\,2^{-n-3}\}$ do not produce jumps of $L_{n}^{-}$
such that $U_{n}^{1}$ holds.

We next assume that $M$ has a jump $r\delta_{(s,y)}$ in $D_{1}\cap
D_{2}^{\mathrm{c}}$. If, additionally, $s\leq t-2^{-\alpha n}\,n^{\rho}$,
then, using Lemma~\ref{L1old} with $\delta=1$, we get
\begin{equation}
\varphi_{n,-}(y)=p_{t-s}^{\alpha}(y)-p_{t-s}^{\alpha}(y-c\,2^{-n-1}%
)\leq2^{1-n}(t-s)^{-2/\alpha}.
\end{equation}
\textrm{F}rom this bound and (\ref{n9}) we obtain
\begin{gather}
\Delta L_{n}^{-}[r\delta_{(s,y)}]\leq C\,2^{-n}(t-s)^{-2/\alpha+\frac
{1}{1+\beta}}\log^{\frac{1}{1+\beta}+q}\bigl((t-s)^{-1}\bigr)|y|^{\frac
{1}{1+\beta}}\log^{\frac{1}{1+\beta}+q}\bigl(|y|^{-1}\bigr)\nonumber\\
\leq C\,2^{-n}(t-s)^{(\frac{1+\alpha}{1+\beta}-2)/\alpha}\log^{\frac{2+\xi
}{1+\beta}+2q}\bigl((t-s)^{-1}\bigr).
\end{gather}
Using the assumption $t-s\geq2^{-\alpha n}\,n^{\rho}$, we arrive at the
inequality
\begin{equation}
\Delta L_{n}^{-}[r\delta_{(s,y)}]\leq C\,2^{-\bar{\eta}_{\mathrm{c}}%
n}\,n^{-\rho(1-\bar{\eta}_{\mathrm{c}})/\alpha+\frac{2+\xi}{1+\beta}+2q}.
\end{equation}
\textrm{F}rom this we see that if we choose $\rho\geq\frac{\alpha
(\xi+2+2q(1+\beta))}{(1+\beta)(1-\bar{\eta}_{\mathrm{c}})}$ then the jumps of
$L_{n}^{-}$ are bounded by $C\,2^{-\bar{\eta}_{\mathrm{c}}n}$, and hence
$U_{n}^{1}$ does not occur.

If $M$ has a jump in $D_{1}\cap D_{2}^{\mathrm{c}}$ at time $s\geq
t-2^{-\alpha n}\,n^{-\rho}$ , then, using (\ref{n9}) and the bound
\begin{equation}
\varphi_{n,-}(y)=p_{t-s}^{\alpha}(y)-p_{t-s}^{\alpha}(y-c\,2^{-n-1})\leq
p_{t-s}^{\alpha}(0)\leq C\,(t-s)^{-1/\alpha},
\end{equation}
we get for $y\in\left(  -(t-s)^{1/\alpha}\log^{\xi}\bigl((t-s)^{-1}%
\bigr),\ c\,2^{-n-3}\right)  $ and $t-s\leq2^{-\alpha n}\,n^{-\rho}$ the
inequality
\begin{gather}
\Delta L_{n}^{-}[r\delta_{(s,y)}]\leq C\,(t-s)^{(\frac{1+\alpha}{1+\beta
}-1)/\alpha}\log^{\frac{2+\xi}{1+\beta}+2q}\bigl((t-s)^{-1}\bigr)\nonumber\\
\leq C\,2^{-\bar{\eta}_{\mathrm{c}}n}\,n^{-\rho(\bar{\eta}_{\mathrm{c}}%
/\alpha)+\frac{2+\xi}{1+\beta}+2q}.
\end{gather}
Choosing $\rho\geq\frac{\alpha(\xi+2+2q(1+\beta))}{(1+\beta)\bar{\eta
}_{\mathrm{c}}}$, we conclude that $\Delta L_{n}^{-}[r\delta_{(s,y)}]\leq
C\,2^{-\bar{\eta}_{\mathrm{c}}n}$, and again $U_{n}^{1}$ does not occur.

Finally, it remains to consider the jumps of $M$ in $D_{1}\cap D_{2}\cap
D_{3\,}^{\mathrm{c}}.$\thinspace\ If the value of the jump does not exceed
$(t-s)^{\frac{\alpha+1}{\alpha(1+\beta)}}\log^{\frac{1}{1+\beta}-2\varepsilon
}\bigl((t-s)^{-1}\bigr)$, then it follows from Lemma~\ref{L1old} with
$\delta=\bar{\eta}_{\mathrm{c}}$ that
\begin{equation}
\Delta L_{n}^{-}[r\delta_{(s,y)}]\leq C\,2^{-\bar{\eta}_{\mathrm{c}}n}%
\log^{\frac{1}{1+\beta}-2\varepsilon}\bigl((t-s)^{-1}\bigr).
\end{equation}
Then, on $D_{2}$, that is, for $t-s>2^{-\alpha n}\,n^{-\rho}$,
\begin{equation}
\Delta L_{n}^{-}[r\delta_{(s,y)}]\leq C\,2^{-\bar{\eta}_{\mathrm{c}}%
n}\,n^{\frac{1}{1+\beta}-2\varepsilon}. \label{n11}%
\end{equation}
Furthermore, if $y<-(t-s)^{1/\alpha}$ and the value of the jump is less than
$\bigl(|y|(t-s)\bigr)^{\frac{1}{1+\beta}}\log^{\frac{1}{1+\beta}-2\varepsilon
}\bigl((t-s)^{-1}\bigr)$, then, using (\ref{n8}), we get
\begin{gather}
\Delta L_{n}^{-}[r\delta_{(s,y)}]\leq C\,2^{-\bar{\eta}_{\mathrm{c}}%
n}(t-s)^{1-\bar{\eta}_{\mathrm{c}}/\alpha}\log^{\frac{1}{1+\beta}%
-2\varepsilon}\bigl((t-s)^{-1}\bigr)|y|^{-\alpha-1+\frac{1}{1+\beta}%
}\nonumber\\
\leq C\,2^{-\bar{\eta}_{\mathrm{c}}n}\log^{\frac{1}{1+\beta}-2\varepsilon
}\bigl((t-s)^{-1}\bigr).
\end{gather}
Then, on $D_{2}$, that is, for $t-s>2^{-\alpha n}\,n^{-\rho}$,
\begin{equation}
\Delta L_{n}^{-}[r\delta_{(s,y)}]\leq C\,2^{-\bar{\eta}_{\mathrm{c}}%
n}\,n^{\frac{1}{1+\beta}-2\varepsilon}. \label{n12}%
\end{equation}
By (\ref{n11}) and (\ref{n12}), we see that the jumps of $M$ in $D_{1}\cap
D_{2}\cap D_{3}^{\mathrm{c}}$ do not produce jumps such that $U_{n}^{1}$
holds. Combining all the above we conclude that to have $\Delta L_{n}%
^{-}[r\delta_{(s,y)}]>C\,2^{-\bar{\eta}_{\mathrm{c}}n}\,n^{\frac{1}{1+\beta
}-2\varepsilon}$ it is necessary to have a jump in $D_{1}\cap D_{2}\cap
D_{3\,}.$\thinspace\ Thus, the proof is finished.
\end{proof}

\begin{proof}
[Proof of Lemma $\ref{n.L5'}$]In view of the Lemmas~\ref{n.L6} and \ref{n.L7},
it suffices to show that
\begin{equation}
\lim_{N\uparrow\infty}\sum_{n=N}^{\infty}\mathbf{P}\!\left(  _{\!_{\!_{\,}}%
}E_{n}(\rho,\xi)\cap A^{\varepsilon}\cap D_{\theta}\right)  =0. \label{4.70}%
\end{equation}
The intensity of the jumps in $D$ [the set defined in~(\ref{setD}) and
satisfying conditions in $E_{n}(\rho,\xi)$\thinspace] is given by
\begin{align}
\int_{t-2^{-\alpha n}n^{\rho}}^{t-2^{-\alpha n}n^{-\rho}}\mathrm{d}s  &
\int_{|y|\leq(t-s)^{1/\alpha}\log^{\xi}\bigl((t-s)^{-1}\bigr)}X_{s}%
(\mathrm{d}y)\label{n13}\\
&  \qquad\qquad\qquad\qquad\qquad\frac{\log^{2\varepsilon(1+\beta
)-1}\bigl((t-s)^{-1}\bigr)}{(t-s)\max\bigl\{(t-s)^{1/\alpha},|y|\bigr\}}%
\,.\nonumber
\end{align}
Since in (\ref{4.70}) we are interested in a limit as $N\uparrow\infty,$ we
may assume that $n$ is such that $\,(t-s)^{1/\alpha}\log^{\xi}\bigl((t-s)^{-1}%
\bigr)\leq1$\thinspace\ for $\,s\geq t-2^{-\alpha n}n^{\rho}.$\thinspace\ We
next note that
\begin{align}
&  \int_{|y|\leq(t-s)^{1/\alpha}}\frac{X_{s}(\mathrm{d}y)}{\max
\bigl\{(t-s)^{1/\alpha},|y|\bigr\}}\\
&  =\,(t-s)^{-1/\alpha}X_{s}\bigl((-(t-s)^{1/\alpha},\,(t-s)^{1/\alpha
})\bigr)\,\leq\ \theta^{-1}\nonumber
\end{align}
on $D_{\theta\,}$. Further, for every $j\geq1$ satisfying $\,j\leq\log^{\xi
}\bigl((t-s)^{-1}\bigr),$
\begin{align}
&  \int_{j(t-s)^{1/\alpha}\leq|y|\leq(j+1)(t-s)^{1/\alpha}}\frac
{X_{s}(\mathrm{d}y)}{\max\bigl\{(t-s)^{1/\alpha},|y|\bigr\}}\\
&  \leq j^{-1}(t-s)^{-1/\alpha}X_{s}\left(  \left\{  y:\;j(t-s)^{1/\alpha}%
\leq|y|\leq(j+1)(t-s)^{1/\alpha}\right\}  \right)  .\nonumber
\end{align}
Since the set $\,\left\{  y:\;j(t-s)^{1/\alpha}\leq|y|\leq(j+1)(t-s)^{1/\alpha
}\right\}  \,\ $is the union of two balls with radius $\,\frac{1}%
{2}\,(t-s)^{-1/\alpha}\,\ $and centers in $B_{2}(0),$ we can apply
Lemma~\ref{n.L2} with $c=1$ to get%
\begin{equation}
\int_{j(t-s)^{1/\alpha}\leq|y|\leq(j+1)(t-s)^{1/\alpha}}\frac{X_{s}%
(\mathrm{d}y)}{\max\bigl\{(t-s)^{1/\alpha},|y|\bigr\}}\,\leq\,2\,\theta
^{-1}j^{-1}%
\end{equation}
on $D_{\theta\,}.$\thinspace\ As a result, on the event $D_{\theta}$ we get
the inequality%
\begin{gather}
\int_{|y|\leq(t-s)^{1/\alpha}\log^{\xi}\left(  (t-s)^{-1}\right)  }%
\!X_{s}(\mathrm{d}y)\,\frac{1}{\max\bigl\{(t-s)^{1/\alpha},|y|\bigr\}}%
\nonumber\\
\leq C\theta^{-1}\log\!\Big(\left\vert \log\bigl((t-s)^{-1}\bigr)\right\vert
\!\Big).
\end{gather}
Substituting this into (\ref{n13}), we conclude that the intensity of the
jumps is bounded by
\begin{equation}
C\theta^{-1}\int_{t-2^{-\alpha n}n^{\rho}}^{t-2^{-\alpha n}n^{-\rho}%
}\mathrm{d}s\ \frac{\log^{2\varepsilon(1+\beta)-1}\Big((t-s)^{-1}\log
\log\bigl((t-s)^{-1}\bigr)\Big)}{(t-s)}\,.
\end{equation}
Simple calculations show that the latter expression is less than
\begin{equation}
C\theta^{-1}n^{2\varepsilon(1+\beta)-1}\log^{1+2\varepsilon(1+\beta)}n.
\end{equation}
Consequently, since $E_{n}(\rho,\xi)$ holds when there are two jumps in $D$,
we have
\begin{equation}
\mathbf{P}\!\left(  _{\!_{\!_{\,}}}E_{n}(\rho,\xi)\cap A^{\varepsilon}\cap
D_{\theta}\right)  \leq C\theta^{-2}n^{4\varepsilon(1+\beta)-2}\log
^{2+4\varepsilon(1+\beta)}n.
\end{equation}
Because $\varepsilon<1/8\leq1/4(1+\beta)$, the sequence $\mathbf{P}\!\left(
_{\!_{\!_{\,}}}E_{n}(\rho,\xi)\cap A^{\varepsilon}\cap D_{\theta}\right)  $ is
summable, and the proof of the lemma is complete.
\end{proof}

\noindent\textbf{Acknowledgement. }We would like to thank an anonymous referee
for encouraging us to prove the optimality of the H\"{o}lder exponent
$\bar{\eta}_{\mathrm{c}}$ [Theorem~\ref{T.prop.dens.fixed}(b)] and also for
suggesting to use a non-classical discretization in the final part of the
proof of Theorem~\ref{T.prop.dens.fixed}(a). Moreover, we thank two anonymous
referees for a careful reading of our revision manuscript.

{\small
\bibliographystyle{alpha}
\bibliography{bibtex,bibtexmy}

\newcommand{\noopsort}[1]{}
\begin{thebibliography}{DPRZ01}

\bibitem[DPRZ01]{DemboPeresRosenZeitouni2001}
A.~Dembo, Y.~Peres, J.~Rosen, and O.~Zeitouni.
\newblock Thick points for planar {B}rownian motion and the
  {E}rd{\"o}s-{T}aylor conjecture on random walk.
\newblock {\em Acta Math.}, 186:239--270, 2001.

\bibitem[Dur09]{Durand2009}
A.~Durand.
\newblock Singularity sets of {L}\'evy processes.
\newblock {\em Probab. Theory Related Fields}, 143(3-4):517--544, 2009.

\bibitem[Fle88]{Fleischmann1988.critical}
K.~Fleischmann.
\newblock Critical behavior of some measure-valued processes.
\newblock {\em Math. Nachr.}, 135:131--147, 1988.

\bibitem[FMW10]{FleischmannMytnikWachtel2010.optimal.Ann}
K.~Fleischmann, L.~Mytnik, and V.~Wachtel.
\newblock Optimal local {H}\"older index for density states of superprocesses
  with $(1+\beta)$-branching mechanism.
\newblock {\em Ann. Probab.}, 38(3):1180--1220, {\noopsort{a}}2010.

\bibitem[HT00]{HuTaylor2000}
X.~Hu and S.J. Taylor.
\newblock Multifractal structure of a general subordinator.
\newblock {\em Stochastic Process. Appl.}, 88:245--258, 2000.

\bibitem[Jaf99]{Jaffard1999}
S.~Jaffard.
\newblock The multifractal nature of {L}\'evy processes.
\newblock {\em Probab. Theory Related Fields}, 114:207--227, 1999.

\bibitem[Jaf00]{Jaffard2000}
S.~Jaffard.
\newblock On {L}acunary wavelet series.
\newblock {\em Ann. Appl. Probab.}, 10(1):313--329, 2000.

\bibitem[Jaf04]{Jaffard2004}
S.~Jaffard.
\newblock Wavelet techniques in multifractal analysis.
\newblock {\em Proc. Symp. Pure Math.}, 72(2):91--151, 2004.

\bibitem[JM96]{JaffardMeyer1996}
S.~Jaffard and Y.~Meyer.
\newblock Wavelet methods for pointwise regularity and local oscillations of
  functions.
\newblock {\em Memb. Am. Math. Soc.}, 123:587, 1996.

\bibitem[KM05]{KlenkeMoerters2005}
A.~Klenke and P.~M{\"o}rters.
\newblock The multifractal spectrum of {B}rownian intersection local time.
\newblock {\em Ann. Probab.}, 33:1255--1301, 2005.

\bibitem[LGP95]{LeGallPerkins1995}
J.-F. Le~Gall and E.A. Perkins.
\newblock The {H}ausdorff measure of the support of two-dimensional
  super-{B}rownian motion.
\newblock {\em Ann. Probab.}, 23(4):1719--1747, 1995.

\bibitem[MP03]{MytnikPerkins2003}
L.~Mytnik and E.~Perkins.
\newblock Regularity and irregularity of $(1+\beta )$-stable super-{B}rownian
  motion.
\newblock {\em Ann. Probab.}, 31(3):1413--1440, 2003.

\bibitem[MS04]{MoertersShieh2004}
P.~M{\"o}rters and N.R. Shieh.
\newblock On the multifractal spectrum of the branching measure on a
  {G}alton-{W}atson tree.
\newblock {\em J. Appl. Probab.}, 41:1223--1229, 2004.

\bibitem[PT98]{PerkinsTaylor1998}
E.A. Perkins and S.J. Taylor.
\newblock The multifractal structure of super-{B}rownian motion.
\newblock {\em Ann. Inst. H. Poincar\'e Probab. Statist.}, 34(1):97--138, 1998.

\end{thebibliography}
}

{\small
\tableofcontents
}%

%

\end{document}